\documentclass[a4paper,11pt,twoside]{amsart}
\usepackage[left=1.0in,right=1.0in]{geometry}
\usepackage[all]{xy} 
\usepackage{amsmath, amssymb, amsfonts, latexsym, mdwlist, amsthm, amscd}
\usepackage{subfig}
\usepackage{graphicx}
\usepackage{wrapfig}

\usepackage[bookmarks, colorlinks, breaklinks, pdftitle={},
pdfauthor={}]{hyperref}
\hypersetup{linkcolor=blue,citecolor=blue,filecolor=black,urlcolor=blue}

\usepackage{tikz}
\usetikzlibrary{calc,cd,trees,positioning,arrows,chains,shapes.geometric,%
    decorations.pathreplacing,decorations.pathmorphing,shapes,%
    matrix,shapes.symbols}

\tikzset{
>=stealth',
  punktchain/.style={
    rectangle,
    rounded corners,
    draw=black, thick,
^{}    
    minimum height=3em,
    text centered,
    on chain},
  line/.style={draw, thick, <-},
  element/.style={
    tape,
    top color=white,
    bottom color=blue!50!black!60!,
    minimum width=8em,
    draw=blue!40!black!90, very thick,
    text width=10em,
    minimum height=3.5em,
    text centered,
    on chain},
  every join/.style={->, thick,shorten >=1pt},
  decoration={brace},
  tuborg/.style={decorate},
  tubnode/.style={midway, right=2pt},
}

\usepackage{paralist}
\setdefaultenum{(a)}{(i)}{}{}
\usepackage{enumitem} 

\usepackage{leftidx}
\usepackage{verbatim}

\def\cA{\mathcal{A}}
\def\cB{\mathcal{B}}
\def\cC{\mathcal{C}}
\def\cD{\mathcal{D}}
\def\cE{\mathcal{E}}
\def\cF{\mathcal{F}}

\def\cG{\mathcal{G}}

\def\cK{\mathcal{K}}
\def\cL{\mathcal{L}}

\def\cT{\mathcal{T}}

\newcommand{\mor}[1][]{\xrightarrow{#1}}
\newcommand{\isomor}{\mor[\sim]}

\def\P{\ensuremath{\mathbb{P}}}

\def\Z{\ensuremath{\mathbb{Z}}}

\def\ev{\mathop{\mathrm{ev}}\nolimits}

\def\Ext{\mathop{\mathrm{Ext}}\nolimits}

\def\Hom{\mathop{\mathrm{Hom}}\nolimits}

\def\RHom{\mathop{\mathbf{R}\mathrm{Hom}}\nolimits}
\def\id{\mathop{\mathrm{id}}\nolimits}

\def\Num{\mathop{\mathrm{Num}}\nolimits}
\def\NS{\mathop{\mathrm{NS}}\nolimits}

\def\MG13{\ensuremath{{\mathcal M}_{\Gamma_1(3)}}}
\def\tildeMG13{\ensuremath{\widetilde{\mathcal M}_{\Gamma_1(3)}}}

\def\into{\ensuremath{\hookrightarrow}}
\def\onto{\ensuremath{\twoheadrightarrow}}

\def\Db{\mathrm{D}^{\mathrm{b}}}

\makeatletter
\newtheorem*{rep@theorem}{\rep@title}
\newcommand{\newreptheorem}[2]{%
\newenvironment{rep#1}[1]{%
 \def\rep@title{#2 \ref{##1}}%
 \begin{rep@theorem}}%
 {\end{rep@theorem}}}
\makeatother

\newtheorem{Thm}{Theorem}[section]
\newreptheorem{Thm}{Theorem}
\newtheorem{Prop}[Thm]{Proposition}

\newtheorem{Lem}[Thm]{Lemma}

\newreptheorem{Cor}{Corollary}
\newtheorem{Con}[Thm]{Conjecture}
\newreptheorem{Con}{Conjecture}

\newtheorem{thm-int}{Theorem}

\newtheorem{ThmInt}{Theorem}[section]

\theoremstyle{definition}
\newtheorem{Def-s}[Thm]{Definition}
\newtheorem{Def}[Thm]{Definition}
\newtheorem{Set}[Thm]{Setup}
\newtheorem{Rem}[Thm]{Remark}

\newtheorem{Ex}[Thm]{Example}

\numberwithin{equation}{section}

\def\C{\ensuremath{\mathbb{C}}}

\def\P{\ensuremath{\mathbb{P}}}

\def\Z{\ensuremath{\mathbb{Z}}}

\def\AA{\ensuremath{\mathcal A}}
\def\BB{\ensuremath{\mathcal B}}

\def\DD{\ensuremath{\mathcal D}}
\def\EE{\ensuremath{\mathcal E}}

\def\GG{\ensuremath{\mathcal G}}

\def\KK{\ensuremath{\mathcal K}}
\def\LL{\ensuremath{\mathcal L}}

\def\OO{\ensuremath{\mathcal O}}

\def\TT{\ensuremath{\mathcal T}}

\newcommand{\Ku}[1][]{\mathcal{K}u({#1})}

\def\llra{\hbox to 10mm{\rightarrowfill}}
\def\lllra{\hbox to 15mm{\rightarrowfill}}

\def\llla{\hbox to 10mm{\leftarrowfill}}
\def\lllla{\hbox to 15mm{\leftarrowfill}}

\def\K{\mathbb K}

\DeclareMathOperator{\Cone}{Cone}

\def\onto{\ensuremath{\twoheadrightarrow}}

\def\llra{\hbox to 10mm{\rightarrowfill}}
\def\lllra{\hbox to 15mm{\rightarrowfill}}
\def\Ku{\mathcal{K}\!u}

\newcommand{\ignore}[1]{}

\begin{document}
\author{}

\title{A Refined Derived Torelli Theorem for Enriques Surfaces}

\author[C.~Li, H.~Nuer, P.~Stellari, and X.~Zhao]{Chunyi Li, Howard Nuer, Paolo Stellari, and Xiaolei Zhao}

\address{C.L.: Mathematics Institute (WMI), University of Warwick, Coventry,
CV4 7AL, United Kingdom.}
\email{C.Li.25@warwick.ac.uk}
\urladdr{\url{https://sites.google.com/site/chunyili0401/}}

\address{H.N.: Department of Mathematics, Statistics, and Computer Science,
University of Illinois at Chicago,
851 S. Morgan Street
Chicago, IL 60607, USA.}
\email{hnuer@uic.edu}
\urladdr{\url{https://sites.google.com/site/howardnuermath/home}}

\address{P.S.: Dipartimento di Matematica ``F.~Enriques'', Universit{\`a} degli Studi di Milano, Via Cesare Saldini 50, 20133 Milano, Italy.}
\email{paolo.stellari@unimi.it}
\urladdr{\url{https://sites.unimi.it/stellari}}

\address{X.Z.: Department of Mathematics,
University of California, Santa Barbara,
6705 South Hall
Santa Barbara, CA 93106, USA.}
\email{xlzhao@math.ucsb.edu}
\urladdr{\url{https://sites.google.com/site/xiaoleizhaoswebsite/}}

\thanks{C.~L.~ was an Early Career Fellow supported by the Leverhulme Trust, and a University Research Fellow supported by the Royal Society. H.~N.~was partially supported by the NSF postdoctoral fellowship DMS-1606283, by the NSF RTG grant DMS-1246844, and by the NSF FRG grant DMS-1664215.
P.~S.~was partially supported by the ERC Consolidator Grant ERC-2017-CoG-771507-StabCondEn, by the research projects PRIN 2017 ``Moduli and Lie Theory'' and FARE 2018 HighCaSt (grant number R18YA3ESPJ), and by the International Associated Laboratory LIA LYSM (in cooperation with AMU, CNRS, ECM and INdAM). X.~Z.~ was partially supported by the Simons Collaborative Grant 636187.}

\keywords{Enriques surfaces, derived categories, Torelli theorem}

\subjclass[2010]{18E30, 14J28, 14F05}

\begin{abstract}
We prove that two general Enriques surfaces defined over an algebraically closed field of characteristic different from $2$ are isomorphic if their Kuznetsov components are equivalent. We apply the same techniques to give a new simple proof of a conjecture by Ingalls and Kuznetsov relating the derived categories of the blow-up of general Artin\textendash Mumford quartic double solids and of the associated Enriques surfaces.

This paper originated from one of the problem sections at the workshop \emph{Semiorthogonal decompositions, stability conditions and sheaves of categories}, Universit\'e de Toulouse, May 2--5, 2018. 
\end{abstract}

\maketitle

\setcounter{tocdepth}{1}
\tableofcontents

\section{Introduction}

An \emph{Enriques surface} is a smooth projective surface $X$ with $2$-torsion dualizing sheaf $\omega_X$ and such that $H^1(X,\OO_X)=0$. In this paper, we assume that all varieties are defined over an algebraically closed field $\K$ of characteristic different from $2$. Under this additional assumption, the above definition is equivalent to asking that $X$ is the quotient of a K3 surface by a fixed-point-free involution.

Because of their construction, Enriques surfaces inherit certain properties from their K3 cousins.  Of interest to us here is an important Hodge theoretic feature that they share: their period maps are injective. In other words, in characteristic zero, Hodge-theoretic Torelli theorems hold for both Enriques and K3 surfaces. More precisely, the formulation of the Torelli theorem for Enriques surfaces involves the weight-$2$ Hodge structure on the second integral cohomology group of the K3 surfaces which are their double covers. 

On the other hand, the bounded derived category of coherent sheaves $\Db(X)$ has very different behaviour depending on whether $X$ is a K3 surface or an Enriques surface. Indeed, in the first case, there might be numerous (albeit finitely many) non-isomorphic K3 surfaces with derived category equivalent to $\Db(X)$ (see \cite{Og,St} in characteristic zero and \cite{LM} in positive characteristic).  Moreover, $\Db(X)$ is indecomposable, i.e.\ it does not contain proper non-trivial admissible subcategories.

When $X$ is an Enriques surface, however, $\Db(X)$ uniquely determines $X$ up to isomorphism (see \cite{BM01} in characteristic zero and \cite{HLT17} for the fields $\K$ of odd positive characteristic).  In analogy to the Hodge theoretic result, we refer to this fact as the \emph{Derived Torelli Theorem} for Enriques surfaces: two Enriques surfaces $X_1$ and $X_2$, defined over a field $\K$ as above, are isomorphic if and only if $\Db(X_1)\cong\Db(X_2)$.  Moreover, we will recall in Section \ref{subsec:Enriques} that for an Enriques surface $X$ there is a collection $\cL:=\{ L_1,\dots,L_{10}\}$ of $10$ exceptional bundles. If $X$ is generic, then it was shown in \cite{Zu} that there is a collection $\cL$ satisfying the additional property that there exists an isomorphism of graded vector spaces
\begin{equation}\label{cond} \tag{$\star$}
\RHom(L_i,L_j)=\K^{\delta_{ij}}.
\end{equation}
In other words, the set $\LL$ is a completely orthogonal exceptional collection.  We then define the \emph{Kuznetsov component} of $X$ to be the admissible subcategory $\Ku(X,\cL):=\langle\LL\rangle^{\perp}$, giving the (non-trivial) semiorthogonal decomposition 
$$\Db(X)=\langle \Ku(X,\cL), \cL \rangle.$$

This decomposition and the properties of $\Ku(X,\cL)$ have been extensively studied in \cite{KI15,Kuz18}. The aim of this paper is to investigate further how much of the geometry of $X$ is encoded by $\Ku(X,\cL)$. The following result should be thought of as a \emph{refined} version of the Derived Torelli Theorem for Enriques surfaces mentioned above.

\begin{ThmInt}\label{thm:derived_torreli}
Let $X_1$ and $X_2$ be Enriques surfaces. Assume that, for $i=1,2$, there exist semiorthogonal decompositions
\begin{equation*}\label{cond2}
\Db(X_i)=\langle \Ku(X_i,\cL_i),\cL_i\rangle
\end{equation*}
and an exact equivalence $\mathsf{F}\colon \Ku(X_1,\cL_1)\xrightarrow{\sim}\Ku(X_2,\cL_2)$ of Fourier\textendash Mukai type, where $\LL_i$ is a collection of 10 bundles satisfying \eqref{cond}. Then $X_1\cong X_2$.
\end{ThmInt}

The terminology used above will be clarified in Section \ref{sec:SOEnriques} while the proof of this result will be carried out in Section \ref{subsec:mainthm1}. More precisely, the result is  an easy consequence of the more general Theorem \ref{thm:gen}.

The essence of the proof is in showing that the equivalence $\mathsf{F}$ can be extended to an equivalence $\Db(X_1)\cong\Db(X_2)$ by means of Propositions \ref{prop:extension} and \ref{prop:extension2} and by adding step by step the $10$ exceptional objects. Then the Derived Torelli Theorem implies our result. The way we obtain the latter global equivalence is by studying and classifying the so-called $3$-spherical objects in $\Ku(X_i,\cL_i)$. This is done in Section \ref{sec:3spherical}. Note that the reason why we have to assume that the characteristic of $\K$ is not $2$ is because the Derived Torelli Theorem for Enriques surfaces is only known to hold under this additional assumption, but the rest of our argument applies even in characteristic $2$ for \emph{classical} Enriques surfaces (see \cite{CD89}).

It is clear that Theorem \ref{thm:derived_torreli} has a trivial converse that, for the sake of completeness, we include in Theorem \ref{thm:derived_torreligen}.

\smallskip

We develop the techniques we use to prove Theorem \ref{thm:derived_torreli} in sufficient generality to give a short proof of a very interesting conjecture by Ingalls and Kuznetsov (see Conjecture \ref{conj:IK}). This will be explained in detail in Section \ref{subsec:AMdoublesolids}, so we content ourselves here with a short summary.

Recall that a general Artin\textendash Mumford quartic double solid is a double cover of the three dimensional projective space, ramified over a quartic symmetroid surface with $10$ nodes, which is the degeneracy locus of a web of quadrics in $\mathbb{P}^3$. The $10$ singular points correspond to quadrics of corank $2$ in the family.

Consider the blow-up $Y'$ of a general Artin\textendash Mumford double solid at its $10$ singular points. It is a classical observation that $Y'$ has an associated Enriques surface $X$ obtained as a quotient of the (desingularization) of the quartic symmetroid. It turns out that $\Db(X)$ has a semiorthogonal decomposition $\Db(X)=\langle\Ku(X,\cL),\cL\rangle$ where $\cL=\{L_1,\dots,L_{10}\}$ satisfies \eqref{cond}.

The main result of \cite{KI15} shows that there is a semiorthogonal decomposition $\Db(Y')=\langle\AA_{Y'},\BB_{Y'}\rangle$, where $\BB_{Y'}$ consists of $12$ exceptional objects. The component $\AA_{Y'}$ has a completely orthogonal exceptional collection $\cG=\{G_1,\dots,G_{10}\}$, and defining $\Ku(Y',\GG):=\langle\GG\rangle^\perp$ gives the semiorthogonal decomposition
\[
\AA_{Y'}=\langle\Ku(Y',\cG),\cG\rangle.
\]
Most importantly, there is an exact equivalence $\Ku(Y',\cG)\cong\Ku(X,\LL)$ of Fourier\textendash  Mukai type (see Theorem \ref{thm:IK}).

The following is our second main result.

\begin{ThmInt}\label{thm:IKConj}
Let $Y'$ be the blow-up of a general Artin\textendash Mumford double solid at its $10$ singular points. Let $X$ be its associated Enriques surface. Then there is an exact equivalence $$\Db(X)\cong\AA_{Y'}$$ which is of Fourier\textendash Mukai type.
\end{ThmInt}

This statement, which is precisely the content of Conjecture \ref{conj:IK} of Ingalls and Kuznetsov \cite{KI15}, was previously proved by Hosono and Takagi \cite{HT} using an intricate argument depending on Homological Projective Duality. Our more precise Theorem \ref{thm:IKConjgen} in Section \ref{subsec:mainthm1} implies Theorem \ref{thm:IKConj} and, given its elementary proof, it provides a simpler proof of this conjecture.

\smallskip

This paper originates from the circle of ideas that stems out of \cite{BMMS}. Indeed, if $Y_1$ and $Y_2$ are cubic threefolds (i.e.\ smooth degree $3$ hypersurfaces in the complex projective space $\P^4$) and $H_i$ is the class of a hyperplane section in Pic$(Y_i)$, then we have semiorthogonal decompositions
\[
\Db(Y_i)=\langle\Ku(Y_i),\OO_{Y_i},\OO_{Y_i}(H_i)\rangle.
\] 
The main result in \cite{BMMS} shows that $Y_1\cong Y_2$ if and only if there is an exact equivalence $\Ku(Y_1)\cong\Ku(Y_2)$. To make the analogy with the paper \cite{BMMS} tighter, we should also mention cubic threefolds and Artin\textendash Mumfold double solids are both Fano threefolds of index $2$ and their derived categories admit very similar semiorthogonal decompositions (see Corollary 3.5 in \cite{KI15}), even though Artin\textendash Mumfold double solids are singular.

If we increase the dimension of the hypersurfaces by one and we consider two cubic fourfolds $W_1$ and $W_2$, then we get semiorthogonal decompositions similar to the one above but with additional exceptional objects $\OO_{W_i}(2H_i)$. In this case, the admissible subcategories $\Ku(W_i)$ behave like the derived category of a K3 surface. Thus, in view of the discussion above, we cannot expect that $W_1\cong W_2$ if and only if there is an exact equivalence $\Ku(W_1)\cong\Ku(W_2)$. And in fact, Huybrechts and Rennemo proved in \cite{HR} that such a statement needs an adjustment: the equivalence $\Ku(W_1)\cong\Ku(W_2)$ has to satisfy some additional and natural compatibility. Contrary to the approach we use in the present paper, the proofs of categorical Torelli for cubic threefolds in \cite{BMMS}, and for cubic fourfolds in \cite[Appendix]{BLMS} and \cite{LPZ}, all make use of Bridgeland stability conditions.

In conclusion, it is worth pointing out that arbitrary semiorthogonal decompositions are in general non-canonical. However, Theorem \ref{thm:derived_torreli} is further evidence of the fact that, when these decompositions originate from geometry, they usually encode important pieces of information.

\section{Semiorthogonal decompositions and an extension result}\label{sec:SOEnriques}

In this section, we briefly recall some basic facts about semiorthogonal decompositions. We also prove an extension result for Fourier\textendash Mukai functors of independent interest and which will be important in this paper. 

\subsection{Generalities}\label{subsec:gen}
In complete generality, let $\cT$ be a triangulated category. A \emph{semiorthogonal decomposition} 
	\begin{equation*}
	\cT = \langle \cD_1, \dots, \cD_m \rangle
	\end{equation*}
	is a sequence of full triangulated subcategories $\cD_1, \dots, \cD_m$ of $\cT$ such that: 
	\begin{enumerate}
		\item $\Hom(F, G) = 0$, for all $F \in \cD_i$, $G \in \cD_j$ and $i>j$;	\item For any $F \in \cT$, there is a sequence of morphisms
		\begin{equation*}  
		0 = F_m \to F_{m-1} \to \cdots \to F_1 \to F_0 = F,
		\end{equation*}
		such that $\pi_i(F):=\mathrm{Cone}(F_i \to F_{i-1}) \in \cD_i$ for $1 \leq i \leq m$. 
	\end{enumerate}
The subcategories $\cD_i$ are called the \emph{components} of the decomposition.

The condition (a) implies that the factors $\pi_i(F)$ in (b) are uniquely determined and functorial. Hence, for all $i=1,\ldots,m$, one can define the \emph{$i$-th projection functor}  $\pi_i\colon\cT\to\cD_i$ for all $F\in\cT$. 

Denote by $\alpha_i\colon\cD_i\hookrightarrow\cT$ the inclusion. We say that $\cD_i$ is \emph{admissible} if $\alpha_i$ has left adjoint $\alpha_{i}^*$ and right adjoint $\alpha_{i}^!$.

Let $\cT_1=\langle\cD_1^1,\cD_2^1\rangle$ and $\cT_2=\langle\cD_1^2,\cD_2^2\rangle$ be two triangulated categories with semiorthogonal decompositions by admissible subcategories $\alpha_{ij}\colon\cD_i^j\hookrightarrow\cT_j$. Following \cite[Section 2.2]{KL}, we can define the \emph{gluing functor}\footnote{Note that our gluing functor differs from the one in \cite{KL} by the shift by $1$. This is harmless and it makes the rest of the discussion easier.}
\[
\Psi_j:=\alpha_{1j}^!\circ\alpha_{2j}\colon\cD^j_2\to\cD_1^j.
\]
Moreover, if $B\in\cD_2^1$, then we denote by $\eta_B\colon\alpha_{11}\Psi_1(B)\to \alpha_{21}(B)$ the counit of adjunction.

The following result will be useful later.

\begin{Lem}\label{lem:crit}
For $i=1,2$, let $\cT_i=\langle\cD_1^i,\cD_2^i\rangle$ be a triangulated category with a semiorthogonal decomposition by admissible subcategories. Let $\mathsf{F}\colon\cT_1\to\cT_2$ be an exact functor with left and right adjoints and such that
\begin{enumerate}
\item[{\rm (a)}]  $\mathsf{F}(\cD^1_j)\subseteq\cD^2_j$ and the induced functors  $\mathsf{F}_j:=\mathsf{F}|_{\cD_j^1}\colon \cD_j^1\to \cD_j^2$ are equivalences for $j=1,2$.
\item[{\rm (b)}] The morphism $\mathsf{F}(\eta_B)$ induces an isomorphism
\[
\xymatrix{
	\Hom\left(\mathsf{F}(\alpha_{11}(A)),\mathsf{F}(\alpha_{11}\Psi_1(B))\right)\ar[rr]^-{\mathsf{F}(\eta_B)\circ-}&&\Hom(\mathsf{F}(\alpha_{11}(A)),\mathsf{F}(\alpha_{21}(B))),
}
\]
for all $A\in\cD_1^1$ and all $B\in\cD_2^1$.
\end{enumerate}
Then $\mathsf{F}$ is an equivalence.
\end{Lem}

To help the reader to keep track of the numerous indices in our notation for categories and functors, we include here a diagram which summarizes them:
\[
\xymatrix{
\cD_1^1\ar[r]^{\alpha_{11}}\ar[d]_{\mathsf{F}_1} & \cT_1\ar[d]^{\mathsf{F}} & \cD_2^1\ar[l]_{\alpha_{21}}\ar[d]^{\mathsf{F}_2}\\
\cD_1^2\ar[r]^{\alpha_{12}} & \cT_2 & \cD_2^2\ar[l]_{\alpha_{22}}.}
\]

\begin{proof}
The objects of $\alpha_{11}(\cD_1^1)$ and $\alpha_{21}(\cD_2^1)$ form a spanning class $\Omega$ for $ \cT_1$ in the sense of \cite{BrEq99}. Hence, by \cite[Theorem 2.3]{BrEq99}, to prove that $\mathsf{F}$ is fully faithful it is enough to show that, for any $A, B \in \Omega$, the natural morphism
\[
\Hom(A,B)\to\Hom(\mathsf{F}(A),\mathsf{F}(B))
\]
induced by $\mathsf{F}$ is bijective.

Since $\mathsf{F}_j$ is an equivalence by (a) and the categories $\cD_1^1$, $\cD_2^1$ are semiorthogonal in $\cT_1$ while $\cD_1^2$, $\cD_2^2$ are semiorthogonal in $\cT_2$, it follows from the definition of a semiorthogonal decomposition that we need only verify that the morphism
\[
\Hom(\alpha_{11}(A),\alpha_{21}(B))\to\Hom(\alpha_{12}\mathsf{F}_1(A),\alpha_{22}\mathsf{F}_2(B))
\]
induced by $\mathsf{F}$ is bijective, for all $A\in\cD_1^1$ and $B\in\cD_2^1$. To this extent, consider the commutative diagram
\begin{equation*}\label{eqn:diagrcomm}
\xymatrix@C-10pt{
	\Hom(\alpha_{11}(A),\alpha_{11}\Psi_1(B))\ar[r]\ar[d]^-{\eta_B\circ-}&\Hom(\mathsf{F}(\alpha_{11}(A)),\mathsf{F}(\alpha_{11}\Psi_1(B)))\ar[d]\ar@{=}[r]&\Hom(\alpha_{12}\mathsf{F}_1(A),\alpha_{12}\mathsf{F}_1(\Psi_1(B)))\ar[d]^-{\mathsf{F}(\eta_B)\circ-}\\
	\Hom(\alpha_{11}(A),\alpha_{21}(B))\ar[r]&\Hom(\mathsf{F}(\alpha_{11}(A)),\mathsf{F}(\alpha_{21}(B)))\ar@{=}[r]&\Hom(\alpha_{12}\mathsf{F}(A),\alpha_{22}\mathsf{F}(B)),
}
\end{equation*}
where the top and bottom rows are obtained by applying $\mathsf{F}$.

Note that the vertical arrows are isomorphisms by adjunction and (b). Since $\mathsf{F}_1$ is an equivalence by (a), the top row is an isomorphism. Thus the bottom one is an isomorphism as well.

The essential surjectivity can be deduced now, as the image of a fully faithful functor is triangulated, and the category $\cT_2$ is generated by $\cD^2_1$ and $\cD^2_2$.
\end{proof}

In the general situation where we have a semiorthogonal decomposition
\[
\cT=\langle\cD_1,\cD_2\rangle,
\]
then $\pi_1$ and $\pi_2$ coincide with the left adjoint $\alpha_1^*$ and and the right adjoint $\alpha_2^!$ of the embeddings $\alpha_i\colon\cD_i\hookrightarrow\cT$. The \emph{left mutation functor} through $\cD_i$ is denoted $\mathsf L_{\cD_i}$ and is defined by the canonical functorial distinguished triangle
\begin{equation}\label{eqn:defofleftmut}
    \alpha_i\alpha_i^!\xrightarrow{\eta_i}\mathsf{id}\rightarrow \mathsf L_{\cD_i},
\end{equation}
where $\eta_i$ denotes the counit of the adjunction.

From now on, let us assume that all categories are linear over a field $\K$. An object $E\in\cT$ is \emph{exceptional} if $\Hom(E,E[p])=0$, for all integers $p\neq0$, and $\Hom(E,E)\cong \K$. A set of objects $\{E_1,\ldots,E_m\}$ in $\cT$ is an \emph{exceptional collection} if $E_i$ is an exceptional object, for all $i$, and $\Hom(E_i,E_j[p])=0$, for all $p$ and all $i>j$.
An exceptional collection $\{E_1,\ldots,E_m\}$ is \emph{orthogonal} if $\Hom(E_i,E_j[p])=0$, for all $i,j=1,\ldots,m$ with $i\neq j$ and for all integers $p$.

If $\TT$ admits a semiorthogonal decomposition $\TT=\langle \DD_1,\DD_2\rangle$ with $\DD_2=\langle E_1,\dots,E_m\rangle$, where $\{E_1,\dots,E_m\}$ is an orthogonal exceptional collection, then the left mutation through $\DD_2$ takes a particularly convenient form: \begin{equation}\label{eqn:LeftMutationExceptional}
\mathsf{L}_{\DD_2}(F)=\Cone\left(\ev:\bigoplus_{1\leq i\leq m,p}\Hom(E_i,F[p])\otimes E_i[-p]\to F\right).
\end{equation} 

\subsection{Extending Fourier\textendash Mukai functors}\label{subsec:extending}

Now let $X$ and $Y$ be smooth projective varieties over $\K$ with admissible embeddings $$\alpha:\AA\into\Db(X)\qquad\text{and}\qquad\beta:\BB\into\Db(Y).$$

Thus $\AA$ (resp. $\BB$) is endowed with a Serre functor $\mathsf{S}_{\AA}$ (resp. $\mathsf{S}_\BB$) by \cite[Lemma 2.8]{KI15}.  Recall that an exact functor $\mathsf{F}\colon\AA\to\BB$ is \emph{of Fourier\textendash Mukai type} if there exists $\mathcal{E}\in\Db(X\times Y)$ such that the composition $\beta\circ \mathsf{F}$ is isomorphic to the restriction 
\[
\Phi_\EE|_{\AA}\colon\AA\to\Db(Y).
\]
Here the exact functor $\Phi_{\EE}$ is given by
\[
\Phi_\mathcal{E}(-):=p_{2*}(\mathcal{E}\otimes p_1^*(-)),
\]
where $p_i$ is the $i$th natural projection. By construction, a morphism $\mathcal{E}_1\to\mathcal{E}_2$ in $\Db(X\times Y)$ induces a morphism
\[
\Phi_{\mathcal{E}_1}(F)\longrightarrow\Phi_{\mathcal{E}_2}(F),
\]
for all $F\in\Db(X)$. Moreover, a distinguished triangle
\[
\mathcal{E}_1\to\mathcal{E}_2\to\mathcal{E}_3
\]
in $\Db(X\times Y)$ induces a distinguished triangle
\[
\Phi_{\mathcal{E}_1}(F)\to\Phi_{\mathcal{E}_2}(F)\to\Phi_{\mathcal{E}_3}(F),
\]
for all $F\in\Db(X)$.

\begin{Rem}\label{rmk:FMfun0onorth}
In the above setup, for a given Fourier\textendash Mukai functor  $\Phi_\cE\colon\Db(X)\to\Db(Y)$  suppose that the induced exact functor $\mathsf{F}:=\Phi_\cE|_{\AA}\colon\AA\to\Db(Y)$ factors through $\BB$. 
By \cite[Theorem 7.1]{Kuz11}, the projection functor onto the admissible subcategory $\AA$ is of Fourier\textendash Mukai type. Thus, if we precompose $\Phi_\cE$ with the projection onto $\AA$, we get a Fourier\textendash Mukai functor $\Phi_{\cE'}\colon\Db(X)\to\Db(Y)$ such that $\Phi_{\cE'}|_{\AA}=\mathsf{F}$ and $\Phi_{\cE'}(\!^\perp\!\AA)=0$.
\end{Rem}

\begin{Rem}\label{rmk:FMfun}
It should be noted that, when $\mathsf{F}$ is an equivalence, \cite[Conjecture 3.7]{Kuz07} would imply that $\mathsf{F}$ is of Fourier\textendash Mukai type. Unfortunately, this conjecture is not known to hold true in the generality needed in this paper. This expectation should be compared with the fact that any full functor $\mathsf{F}\colon\Db(X)\to\Db(Y)$ is of Fourier\textendash Mukai type (see \cite{CS,Or}). As for this, it should be noted that in the seminal paper \cite{Or} such a result was proved under the assumption that $\mathsf{F}$ is fully faithful but in \cite{COS,CS} the assumptions were weakened. In particular, assuming full is enough.
\end{Rem}

We can now explain how to extend a Fourier\textendash Mukai functor to an exceptional object.

\begin{Prop}\label{prop:extension}
Let $\alpha\colon\cA\hookrightarrow\Db(X)$ be an admissible embedding and let $E\in \!^\perp\!\cA$ with counit of adjunction  $\eta\colon\alpha\alpha^!(E)\to E$ .
Let $\Phi_{\cE}:\Db(X)\rightarrow \Db(Y)$ be a Fourier\textendash Mukai functor with the property that $\Phi_\cE(^\perp\!\cA)\cong 0$.  For any object $F$ in $\Db(Y)$ and any morphism $\varphi:\Phi_\cE(\alpha\alpha^!(E))\rightarrow F$, there is an object $\cE_\varphi\in\Db(X\times Y)$ and a morphism $\psi_F:\cE_\varphi\rightarrow \cE$ such that 
\begin{enumerate}
    \item[\rm{(1)}] the induced morphism $\Phi_{\cE_\varphi}|_{\cA}\xrightarrow{\psi_F}\Phi_\cE|_{\cA}$ is an isomorphism;
    \item[\rm{(2)}] $\Phi_{\cE_\varphi}|_{^\perp\!\langle \cA,E\rangle}=0$; and
    \item[\rm{(3)}] there is an isomorphism $\Phi_{\cE_\varphi}(E)\cong F$ such that the diagram \begin{equation}\label{eqn:inextprop}
\xymatrix{
	\Phi_{\cE_\varphi}(\alpha\alpha^!(E))\ar[d]^-{\psi_F(\alpha\alpha^!(E))}\ar[rr]^-{\Phi_{\cE_\varphi}(\eta)}&&	\Phi_{\cE_\varphi}(E) \ar[d]^{\cong} \\
	\Phi_{\cE}(\alpha\alpha^!(E))\ar[rr]^-{\varphi} && F
}
\end{equation}
is commutative.
\end{enumerate}
\end{Prop}

\begin{proof}
Consider the triangle \eqref{eqn:defofleftmut} applied to the object $E$:
\begin{equation}\label{eqn:leftmutofE}
\alpha\alpha^!(E)\xrightarrow{\eta} E\rightarrow \mathsf L_{\AA}(E).
\end{equation}
Set $E':=\mathsf S_{X}\left((\mathsf L_{\AA}(E))^\vee\right)$. If we take the derived dual of \eqref{eqn:leftmutofE} and we apply $\mathsf{S}_{X}$ we get the triangle
\begin{equation}\label{eqn:twistdual}
    \mathsf S_{X}(E^\vee)\rightarrow \mathsf S_{X}\left((\alpha\alpha^!(E))^\vee\right)\stackrel{\gamma}{\longrightarrow} E'[1].
\end{equation}

Next, we observe that, for any $K\in\Db(X\times Y)$, $M\in\Db(X)$ and $N\in\Db(Y)$, we get isomorphisms
\begin{equation}\label{eqn:usefformu}
\begin{split}
    \Hom(\Phi_K(M),N)&=\Hom(p_{2*}(K\otimes p_1^*M),N)\\&\cong \Hom(K\otimes p_1^*M,p_2^!N)\\&\cong \Hom(K,\mathsf S_{X}(M^\vee)\boxtimes N)
\end{split}
\end{equation}
which are functorial in $K$, $M,$ and $N$. Thus, we get a natural morphism
\begin{equation}\label{eqn:ident}
\Hom(\Phi_\cE(\alpha\alpha^! (E)),F)\cong \Hom\left(\cE,\mathsf S_{X}\left((\alpha\alpha^!(E))^\vee\right)\boxtimes F\right)\stackrel{\gamma\boxtimes\id_F}{\longrightarrow}\Hom(\cE,E'\boxtimes F[1]),
\end{equation}
where the latter map is induced by the postcomposition by $\gamma\boxtimes\id_F$ for $\gamma$ as in \eqref{eqn:twistdual}.

Such a morphism of vector spaces takes $\varphi$ as in the assumptions of the proposition to a morphism $\varphi':\cE\rightarrow E'\boxtimes F[1]$. We set $\psi_F:\cE_\varphi\to\EE$ to be the cocone of $\varphi'$ as in the distinguished triangle
\begin{equation}\label{eqn:triofkernels}
\cE_\varphi\xrightarrow{\psi_F}\cE\xrightarrow{\varphi'}E'\boxtimes F[1].
\end{equation}
Clearly, $\psi_F$ induces a natural transformation $\Phi_{\cE_\varphi}\to\Phi_\cE$ which we denote with the same symbol.

Note that ${}^\perp\langle\cA,E\rangle\subseteq{}^\perp\cA$. Then, an easy computation using \eqref{eqn:usefformu} shows that $\mathsf{L}_\cA(E)\in\cA^\perp$ and hence $\Phi_{E'\boxtimes F}(\cA)=0$ and, similarly, that  $\mathsf{L}_\cA(E)\in\langle\cA,E\rangle=({}^\perp\langle\cA,E\rangle)^\perp$ and hence $\Phi_{E'\boxtimes F}({}^\perp\langle\cA,E\rangle)=0$. Therefore
\begin{equation}\label{eqn:vanofcomps}
\Phi_{E'\boxtimes F}(\AA)\cong\Phi_{E'\boxtimes F}(\!^\perp\!\langle\AA,E\rangle)\cong 0\qquad \Phi_{E'\boxtimes F}(E)\cong F.
\end{equation}
Since, by assumption, $\Phi_\cE(^\perp\!\AA)\cong 0$, the first two isomorphisms imply (1) and (2) in the statement. 

To complete the proof of (3) and deduce the commutative diagram \eqref{eqn:inextprop}, we apply the Fourier\textendash Mukai functors whose kernels are the objects in \eqref{eqn:triofkernels} to the triangle in \eqref{eqn:leftmutofE}. We obtain the commutative diagram
\begin{equation}\label{eqn:comdiagPhi}
\xymatrix{
	\Phi_{\cE_\varphi}(\mathsf L_{\AA} (E)[-1])\ar[d]^-{\psi_F(L_{\AA} (E)[-1])}\ar[rr] &&	\Phi_{\cE_\varphi}(\alpha\alpha^!(E)) \ar[rr]^-{\Phi_{\cE_\varphi}(\eta)}\ar[d]^-{\psi_F(\alpha\alpha^!(E))}&&\Phi_{\cE_\varphi}(E) \ar[d]  \\
	\Phi_{\cE}(\mathsf L_{\AA} (E)[-1])\ar[d]^-{\varphi'(\mathsf L_{\AA} (E)[-1])}\ar[rr]^{\epsilon} &&	\Phi_{\cE}(\alpha\alpha^!(E)) \ar[rr]\ar[d] && 0\ar[d] \\
	\Phi_{E'\boxtimes F[1]}(\mathsf L_{\AA} (E)[-1])\ar[rr] && 0	\ar[rr] && F[1] 
}
\end{equation}
with rows and columns which are distinguished triangles. The zero in the bottom row of the diagrams comes from the fact that $\Phi_{E'\boxtimes F}(\alpha\alpha^!(E))\cong 0$ by \eqref{eqn:vanofcomps}.  The zero on the third column depends on the fact that $E\in{}^\perp\cA$ and, by assumption, $\Phi_\mathcal{E}({}^\perp\cA)\cong 0$. In conclusion, the morphisms $\epsilon$ and $\psi_F(\alpha\alpha^!(E))$ are both isomorphisms.

Now we can recast its top row and left column in the following commutative diagram
\begin{equation}\label{eqn:comdiagsmall}
\xymatrix{
	\Phi_{\cE_\varphi}(\mathsf L_{\AA} (E)[-1])\ar@{=}[d]\ar[rr] &&	\Phi_{\cE_\varphi}(\alpha\alpha^!(E)) \ar[rr]^-{\Phi_{\cE_\varphi}(\eta)}\ar[d]^-{\psi_F(\alpha\alpha^!(E))}&&\Phi_{\cE_\varphi}(E)\\
	\Phi_{\cE_\varphi}(\mathsf L_{\AA}(E)[-1])\ar[rr] &&	\Phi_{\cE}(\alpha\alpha^!(E)) \ar[rr]^-{\varphi'\circ\epsilon^{-1}} && 
	\Phi_{E'\boxtimes F[1]}(\mathsf L_{\AA}(E)[-1])
}
\end{equation}
where we write $\varphi'$ as a shorthand for the natural transformation $\varphi'(\mathsf L_{\AA} (E)[-1])$ coming from \eqref{eqn:comdiagsmall}.  Observe that the middle vertical map is an isomorphism by \eqref{eqn:comdiagsmall}, since $\alpha\alpha^!(E)\in\AA$.  Therefore, there is an isomorphism 
\begin{equation}\label{eqn:impdop}
\Phi_{\cE_\varphi}(E)\to\Phi_{E'\boxtimes F[1]}(\mathsf L_{\AA}(E)[-1])\cong F
\end{equation}
making all squares in \eqref{eqn:comdiagsmall} commutative. The right square is precisely \eqref{eqn:inextprop} as the bottom morphism identifies with $\varphi$ under the maps \eqref{eqn:ident} and the isomorphism $\Phi_{E'\boxtimes F[1]}(\mathsf L_{\AA}(E)[-1])\cong\Phi_{E'\boxtimes F[1]}(E[-1])\cong F$.
\end{proof}

Now we consider the case when $\Phi_{\EE}|_{\AA}$ is an equivalence onto its image.

\begin{Prop}\label{prop:extension2}
Let $\alpha\colon\cA\hookrightarrow\Db(X)$ be an admissible embedding and let $E\in \!^\perp\!\cA$ with counit of adjunction  $\eta\colon\alpha\alpha^!(E)\to E$ .
Let $\Phi_{\cE}:\Db(X)\rightarrow \Db(Y)$ be a Fourier\textendash Mukai functor with the property that $\Phi_\cE(^\perp\!\cA)\cong 0$.  Suppose further that
\begin{itemize}
\item[{\rm (a)}] $\Phi_{\EE}|_\AA$ is an equivalence onto an admissible subcategory $\BB$ with embedding $\beta:\BB\into\Db(Y)$, and
\item[{\rm (b)}] there is an exceptional object $F\in\!^\perp\!\BB$ and an isomorphism $\rho:\Phi_\cE(\alpha\alpha^!(E))\isomor\beta\beta^!(F)$.
\end{itemize}
Then the Fourier\textendash Mukai functor $\Phi_{\cE_\varphi}$ restricts to an equivalence between $\langle \cA,E\rangle$ and $\langle \BB,F\rangle$, where $\EE_{\varphi}$ is the object given by Proposition \ref{prop:extension} for $\varphi:=\eta'\circ\rho$ and $\eta':\beta\beta^!(F)\to F$ is the counit of adjunction.
\end{Prop}

\begin{proof}
We apply Lemma \ref{lem:crit}. The assumption (a) of the lemma obviously holds. It is enough to check that for any object $A$ in $\AA$,
\begin{equation}\label{eqn:iss}
\xymatrix{
\Hom(\Phi_{\cE_\varphi}(A),\Phi_{\cE_\varphi}(\alpha\alpha^!(E)))\ar[rrr]^-{\Phi_{\cE_\varphi}(\eta)}&&&\Hom(\Phi_{\cE_\varphi}(A),\Phi_{\cE_\varphi}(E))
}
\end{equation}
is an isomorphism. For this we consider the commutative diagram \eqref{eqn:inextprop} in Proposition \ref{prop:extension} (3), which takes the form 
\[
\xymatrix{
	\Phi_{\cE_\varphi}(\alpha\alpha^!(E))\ar[d]^-{\psi_F(\alpha\alpha^!(E))}\ar[rr]^-{\Phi_{\cE_\varphi}(\eta_{1})} &&	\Phi_{\cE_\varphi}(E) \ar[d]^-{\cong} \\
	\Phi_{\cE}(\alpha\alpha^!(E))\ar[rr]^-{\varphi=\eta'\circ\rho} && F
}
\]
where the vertical arrows are isomorphisms for the same arguments as in the proof of Proposition \ref{prop:extension} (see, in particular, the middle column in \eqref{eqn:comdiagPhi} for the isomorphism on the left and \eqref{eqn:impdop} for the one on the right).

Now the cone $C$ of the morphism at the top is isomorphic to the cone of $\varphi$, which in turn is isomorphic to $\mathsf L_{\BB}(F)\in\BB^\perp$.  It follows that $C\in\BB^\perp$, so $C$ is right orthogonal to $\Phi_{\cE_\varphi}(A)\in\BB$ for any  $A\in \AA$. Therefore, if we apply $\RHom(\Phi_{\cE_\varphi}(A),-)$ to the distinguished triangle
\[
\xymatrix{
\Phi_{\cE_\varphi}(\alpha\alpha^!(E)))\ar[rr]^-{\Phi_{\cE_\varphi}(\eta)} &&	\Phi_{\cE_\varphi}(E)\ar[r]& C
},
\]
we get the isomorphism \eqref{eqn:iss}.
\end{proof}

\subsection{A side remark: a dg category approach}\label{subsect:dgcat}

Even though our approach in this paper is purely triangulated, this section ends with a short discussion concerning an alternative viewpoint via dg categories (the non-expert reader can have a look at \cite{CS07} for a quick introduction to this subject). More precisely, this section should be considered as an indication of how an expert in dg categories might read the results in Section \ref{subsec:extending}. On the other hand, the reader who is not interested in this higher categorical viewpoint can move to Section \ref{sec:geometry}.

Let us first rediscuss semiorthogonal decompositions in terms of dg categories. Assume that $\cT=\langle\cD_1,\cD_2 \rangle$ is a small $\K$-linear \emph{algebraic triangulated category}, i.e.\ there is a small $\K$-linear pre-triangulated dg category $\cB$ and an exact equivalence $\cT\cong H^0(\cB)$, where $H^0(\cB)$ denotes the homotopy category of $\cB$.
Suppose that $\alpha_i\colon\cD_i\hookrightarrow\cT$ is the embedding of an admissible subcategory, for $i=1,2$. Then there exist two pre-triangulated dg subcategories $\mathsf{G}_i\colon\cB_i\hookrightarrow\cB$ such that $H^0(\cB_i)\cong\cD_i$ and $H^0(\mathsf{G}_i)\cong\alpha_i$. As it is explained for example in \cite[Section 4]{KL}, there exist a dg bimodule $\varphi$ (i.e.\ a $\cB_1^\circ\otimes\cB_2$-dg module), a pre-triangulated dg category $\cB_1\times_\varphi\cB_2$, whose definition depends on $\cB_1$, $\cB_2$ and $\varphi$, and an isomorphism $\cB\cong\cB_1\times_\varphi\cB_2$ in $\mathrm{Ho}(\mathbf{dgCat})$. Here $\mathrm{Ho}(\mathbf{dgCat})$ denotes the homotopy category of the category $\mathbf{dgCat}$ of (small) dg categories which are linear over the field $\K$. We will refer to $\cB_1\times_\varphi\cB_2$ as the \emph{gluing of $\cB_1$ and $\cB_2$ along $\varphi$}. By \cite[Corollary 4.5]{KL}, the dg bimodule $\varphi$ yields, at the triangulated level, the gluing functor in Section \ref{subsec:gen}.

Assume that $\cC_1$ and $\cC_2$ are two pre-triangulated dg categories which are obtained by this gluing procedure. In explicit form, $\cC_i=\cC^i_1\times_{\varphi_i}\cC^i_2$.

\begin{Ex}\label{ex:dgcat}
For our applications and in accordance with the discussion and notations in the previous section, we should think of the case $\langle\cA,E\rangle\cong H^0(\cC_1)$ and $\langle\cB,F\rangle\cong H^0(\cC_2)$. The dg subcategories $\cC_j^i$ are taken so that $\cA\cong H^0(\cC^1_1)$ $\cB\cong H^0(\cC^2_1)$, $\langle E\rangle\cong H^0(\cC^1_2)$ and $\langle F\rangle\cong H^0(\cC^2_2)$.
\end{Ex}


Let $\mathsf{J}_1\colon\cC^1_1\to\cC^2_1$ and $\mathsf{J}_2\colon\cC^1_2\to\cC^2_2$ be two isomorphisms in $\mathrm{Ho}(\mathbf{dgCat})$. We would like to conclude that $\mathsf{J}_1$ and $\mathsf{J}_2$ glue to an isomorphism $\mathsf{J}\colon\cC_1\to\cC_2$ in $\mathrm{Ho}(\mathbf{dgCat})$, under some reasonable assumptions.
To this extent, consider the $(\cC_1^1)^\circ\otimes\cC_2^1$-module $\varphi'_1$ obtained from $\varphi_2$ by composing with $\mathsf{J}_1$ and $\mathsf{J}_2$.

The following is a well-known result, which should be thought of as the dg analogue of Propositions \ref{prop:extension} and \ref{prop:extension2}.

\begin{Prop}\label{prop:KL}
	Under the assumptions above, suppose that there is a quasi-isomorphism between $\varphi_1$ and $\varphi'_1$. Then there is an isomorphism $\mathsf{J}\colon\cC_1\to\cC_2$ in $\mathrm{Ho}(\mathbf{dgCat})$ whose restriction to $\cC^1_j$ is $\mathsf{J}_j$, for $j=1,2$.
\end{Prop}

\begin{proof}
	The argument is an easy adaptation of the proof of Propositions 4.11 and 4.14 in \cite{KL}. 
\end{proof}

Roughly, the existence of a quasi-isomorphism between $\varphi_1$ and $\varphi'_1$ means that $\varphi_1$ and $\varphi_2$ are compatible under the action of $\mathsf{J}_1$ and $\mathsf{J}_2$. In general, it is clear that this assumption is not easy to verify.
On the other hand, if we look at the triangulated counterpart of this hypothesis in the simplified situation in Example \ref{ex:dgcat}, it corresponds to the assumptions on $\Phi_\EE$ in Proposition \ref{prop:extension2}. Hence, the reader should interpret the results in the previous section as a way to give a conceptually simple triangulated version of Proposition \ref{prop:KL}. In this case, the conditions under which one glue exact equivalences will be easy to check, see Section \ref{sec:genextresult} for an application of this technique.

It should be noted that the dg approach to semiorthogonal decompositions discussed in this section was put in an even more general categorical setting in \cite{RVdB} (see, in particular, \cite[Section 9.3]{RVdB}).

\section{The geometric setting}\label{sec:geometry}

In this section, we discuss the structure of the derived categories of Enriques surfaces and of Artin\textendash Mumford quartic double solids. The emphasis is on  the analogies which are at the core of this paper.

\subsection{The case of Enriques surfaces}\label{subsec:Enriques}
We work over an algebraically closed field $\K$ of characteristic different from $2$. We refer to \cite{D} for a quick but exhaustive introduction to the geometry of Enriques surfaces. For the moment, we recall several useful facts.

First note that the Serre functor $\mathsf S_X(-):=(-)\otimes\omega_X[2]$ satisfies the property $\mathsf S_X^2=[4]$, but $\mathsf{S}_X\ne[2]$ because the dualizing sheaf $\omega_X:=\OO_X(K_X)$ is non-trivial but $2$-torsion. Note that $\mathrm{NS}(X)_{\mathrm{tor}}\cong\Z/2\Z$ and it is generated by the class of $K_X$. We set $\Num(X):=\mathrm{NS}(X)/\mathrm{NS}(X)_{\mathrm{tor}}$. If $F$ is a line bundle on $X$ and $f$ is its class in $\Num(X)$, then Riemann\textendash Roch takes the simple form
\begin{equation}\label{eqn:RR}
\chi(F)=\frac{1}{2}f^2+1
\end{equation}
for an Enriques surface.

Secondly, recall that an Enriques surface $X$ is called \emph{unnodal} if it contains no smooth rational curves.  It is called \emph{nodal} otherwise. Denote by $\mathbf{M}$ the $10$-dimensional moduli space of Enriques surfaces. Nodal Enriques surfaces form an irreducible divisor in $\mathbf{M}$ (see, for example, \cite[Section 5]{D}).

Now let us make precise a result that was mentioned in the introduction. Indeed, as an application of the stronger statements \cite[Proposition 6.1]{BM01} and \cite[Theorem 1.1]{HLT17}, we get the following result.

\begin{Thm}[Bridgeland\textendash Maciocia, Honigs\textendash Lieblich\textendash Tirabassi]\label{thm:dertordercat}
Let $X$ and $Y$ be smooth projective surfaces defined over an algebraically closed field $\K$ of characteristic different from $2$. If $X$ is an Enriques surface and there is an exact equivalence $\Db(X)\cong\Db(Y)$, then $X\cong Y$.
\end{Thm}

To the best of our knowledge, a similar statement is not known when $\K$ has characteristic $2$, due to the additional complexity of the double cover structure. We should also observe that generalizations of Theorem \ref{thm:dertordercat} are available in the twisted setting (see \cite{AW} when the characteristic is $0$ and \cite{HLT17} in positive characteristic).

\smallskip

As we mentioned in the introduction, the category $\Db(X)$ admits a very nice semiorthogonal decomposition.  To describe it, we must begin with the relationship between Fano polarizations and canonical isotropic 10-sequences.

\begin{Def} 
Let $X$ be an Enriques surface.
\begin{enumerate}
    \item A \emph{Fano polarization} on $X$ is a nef divisor $\Delta$ such that\begin{enumerate}
	\item $\Delta^2=10$;
	\item $\Delta.F\geq 3$ for every nef divisor $F$ with $F^2=0$.
\end{enumerate} 
\item An \emph{isotropic $10$-sequence} in $\Num(X)$ is a set of $10$ vectors $\cF=\{f_1,\dots, f_{10}\}\subset\Num(X)$ such that $f_i.f_j=1-\delta_{ij}$, where $\delta_{ij}$ is Kronecker delta.
\end{enumerate}

\end{Def}

 If $\Delta$ is a Fano polarization, then by \cite[Theorem 2.6]{DM} it determines an isotropic $10$-sequence $\{f_1,\dots, f_{10}\}$ which is unique up to permutation and such that the numerical class $\delta$ of $\Delta$ satisfies
\begin{equation}\label{eqn:FP1}
\delta=\frac{1}{3}(f_1+\dots+f_{10}).
\end{equation}
Even though this is not going to be used later, let us note that the statement above can be made more precise. Indeed, as explained in Sections 2.2 and 2.3 of \cite{DM}, the isotropic $10$-sequence $\{f_1,\dots,f_{10}\}$ yielding \eqref{eqn:FP1} is \emph{canonical}, i.e.\ it contains all the nef representatives of its orbit under the action of the Weyl group generated by the reflections in the classes of the $(-2)$-curves. Moreover, by \cite[Theorem 2.7]{DM}, the converse is also true: given a canonical isotropic $10$-sequence $\{f_1,\dots,f_{10}\}$, there is a Fano polarization $\Delta$ on $X$ whose numerical class $\delta$ satisfies \eqref{eqn:FP1}.

\begin{Rem}\label{rmk:Fanoample}
(i) By \cite[Proposition 1]{DR91}, every Enriques surface possesses a Fano polarization $\Delta$. On the other hand, by \cite[Proposition 2.7]{DM}, a Fano polarization $\Delta$ is ample if and only if the corresponding isotropic $10$-sequence $\{f_1,\dots, f_{10}\}$ is made of nef classes.

(ii) If $X$ is unnodal, an easy application of \cite[Theorem 3.2]{Cos85} shows that $X$ has an ample Fano polarization $\Delta$. A generic nodal Enriques surface has the same property by \cite[Theorem 3.2.2]{Cos} and \cite[Corollary 4.4]{DM}. Thus Enriques surfaces with an ample Fano polarization form an open subset in $\mathbf{M}$ whose complement has codimension at least $2$. 
\end{Rem}

\begin{Ex}\label{ex:EnrC}
Let $X$ be an unnodal Enriques surface defined over $\C$ with an ample Fano polarization $\Delta$. Let $\{f_1,\dots,f_{10}\}$ be the isotropic  $10$-sequence associated to $\Delta$, and let $F_i^{\pm}\in \NS(X)$ be the two preimages of $f_i$ under the quotient map $\NS(X)\onto\Num(X)$. Then for $i=1,\dots,10$, $F_i^+=F_i^-+K_X$ and the linear series $$|2F_i^+|=|2F_i^-|$$ has precisely two double fibers, supported on $F_i^+$ and $F_i^-$. As observed in \cite[Example 2.6]{Zu}, $\{\OO_X(F_i^+)\}_{1\leq i\leq 10}$ is an orthogonal exceptional collection. Each of the objects $\OO_X(F_i^+)$ can be individually changed to $\OO_X(F_i^-)$ just by tensoring by the canonical bundle and the new collection is still an orthogonal exceptional collection by Serre duality. Thus we have at least $2^{10}=1024$ possible choices of orthogonal exceptional collections of line bundles in $\Db(X)$.  Of course, in every case the numerical class of the sum of all ten divisors remains unchanged, equal to $3\delta$.  By \cite[Theorem 3.11]{BP}, there are exactly $2^{13}\cdot3\cdot17\cdot31$ numerical classes of Fano polarizations up to the action of the Weyl group.  So each of these give rise to $1024$ different orthogonal exceptional collections as above.
\end{Ex}

We now want to discuss how to use isotropic $10$-sequences to construct interesting collections of exceptional objects and generalize the picture in Example \ref{ex:EnrC}. The following is probably well-known to experts but we state and prove it here for the sake of completeness.

\begin{Prop}\label{prop:exceptionalcollectionexists}
Let $X$ be an Enriques surface over $\K$ with a Fano polarization $\Delta$. Then $\Db(X)$ contains an admissible subcategory $\cL=\langle\cL_1,\dots,\cL_c\rangle$, where $\cL_1,\dots,\cL_c$ are completely orthogonal admissible subcategories and
\[
\cL_i=\langle L^i_1,\dots,L^i_{n_i}\rangle
\]
where
\begin{itemize}
\item[{\rm (1)}] $L_j^i$ is a line bundle such that $L^i_j=L^i_1\otimes\OO_X(R^i_1+\dots+R^i_{j-1})$ where $R^i_1,\dots,R^i_{j-1}$ is a chain of $(-2)$-curves of $A_{j-1}$ type; 
\item[{\rm (2)}] $\{L^i_1,\dots,L^i_{n_i}\}$ is an exceptional collection; and
\item[{\rm (3)}] $n_1+\dots+n_c=10$.
\end{itemize}
Moreover, if $\Delta$ is ample, then $c=10$, $n_1=\dots=n_{10}=1$ and $\cL$ is generated by an orthogonal exceptional collection $\{L_1,\dots,L_{10}\}$ consisting of line bundles.
\end{Prop}

\begin{proof}
By Remark \ref{rmk:Fanoample} (i), $\Delta$ determines a unique (up to reordering) isotropic $10$-sequence $\cF\subset\Num(X)$. By \cite[Lemma 3.3.1]{CD89} (see also \cite[Section 2]{DM} for an extensive discussion), such an isotropic $10$-sequence can be written as a disjoint union
\begin{equation}\label{eqn:FP2}
\cF=\cF_1\sqcup\dots\sqcup\cF_c,
\end{equation}
where, for $i=1,\dots,c$,
\[
\cF_i=\{f^i_1,\dots,f^i_{n_i}\}
\]
and letting $F_j^{i,\pm}$ be the two lifts of $f_j^i$, we have
\begin{itemize}
\item[{\rm (a)}] $F^{i,\pm}_j=F^{i,\pm}_1+ R^i_1+\dots+R^i_{j-1}$, where $R^i_1,\dots,R^i_{j-1}$ is a chain of $(-2)$-curves of $A_{j-1}$ type;
\item[{\rm (b)}] $f_1^i$ is nef for all $i=1,\dots,c$, and the classes $f^i_j$, for $j=1,\dots,n_i$, are all the elements of $\cF$ which are conjugate to $f^i_1$ under the Weyl group action and ordered in such a way that $f^i_{j+1}-f^i_j=R^i_j$; and
\item[{\rm (c)}] $n_1+\dots+n_c=10$.
\end{itemize}

Let $L_j^i=\OO_X(F_j^{\pm})$. Since $H^s(X,\OO_X)= 0$ for $s=1,2$, all line bundles are exceptional objects.  Moreover, by construction, $\RHom(L^i_j,L^i_k)\cong 0$, for all $i=1,\dots,c$ and all $j>k\in\{1,\dots,n_i\}$ and so $\{L^i_1,\dots,L^i_{n_i}\}$ is an exceptional collection for $i=1,\dots,c$ (see \cite[Theorem B]{HP}). By (a) and (c), conditions (1), (2), and (3) in the statement are satisfied by the $L^i_j$'s. Thus it remains to show that, if we set $\cL_i=\langle L^i_1,\dots,L^i_{n_i}\rangle$, then the block $\cL_i$ is completely orthogonal to $\cL_j$ when $i\neq j\in\{1,\dots,c\}$.

To this extent, it is enough to show that $\RHom(L^j_l,L^k_m)\cong 0$, for all $j\neq k\in\{1,\dots,c\}$, $l=1,\dots,n_j$ and $m=1,\dots,n_k$. But
\[
\Hom(L^j_l,L^k_m)\cong H^0(X,L^k_m\otimes(L^j_l)^\vee)=0
\]
because $f^k_m-f^j_l$ is not effective by (b). Indeed, if $f^k_m-f^j_l$ is effective, then such a class is a $(-2)$-curve. Thus $f^k_m$ and $f^j_l$ would be in the same orbit under the action of the Weyl group, contradicting the fact that $j\neq k$.

Likewise, $\Hom(L^k_m,L^j_l\otimes\OO_X(K_X))= 0$ because $f^j_l-f^k_m$ is not effective. Therefore, it follows from Serre duality that 
\[
\Ext^2(L^j_l,L^k_m)^\vee\cong\Hom(L^k_m,L^j_l\otimes\OO_X(K_X))\cong 0.
\]
By Riemann\textendash Roch \eqref{eqn:RR}, we have $\chi(L^j_l,L^k_m)=\chi(L^k_m\otimes(L^j_l)^\vee)=\frac{1}{2}(f^k_m-f^j_l)^2+1=0$, so we get orthogonality as claimed.

If $\Delta$ is ample then, by Remark \ref{rmk:Fanoample} (i) all the vectors in $\cF$ are nef. By \cite[Lemma 3.3.1]{CD89}, we get $c=10$ in \eqref{eqn:FP2} and $n_i=1$ for $i=1,\dots,10$.  It follows from the argument above that $\cL=\{L_1,\dots,L_{10}\}$ is an orthogonal exceptional collection\footnote{Note that the case when $X$ has an ample polarization was also treated in  \cite[Lemma 3.7]{MST}}. This concludes the proof.
\end{proof}

\begin{Rem}\label{rmk:noexc}
(i) Assume that $X$ has an ample Fano polarization so that the admissible subcategory $\cL$ in Proposition \ref{prop:exceptionalcollectionexists} is generated by a collection of $10$ orthogonal exceptional objects $L_1,\dots,L_{10}$. It is clear that the choice of the $L_i$'s is not unique. For example, we can replace any $L_i$ by $L_i\otimes\OO_X(K_X)$. Thus, as in Example \ref{ex:EnrC}, we have at least $2^{10}=1024$ choices of $10$ orthogonal exceptional objects in $\Db(X)$.

(ii) In \cite{Ko86}, the author provides e xamples of Enriques surfaces for which no Fano polarization $\Delta$ has a $c$ (as in Proposition \ref{prop:exceptionalcollectionexists}) equal to $10$. By Remark \ref{rmk:Fanoample} (i), it follows that $X$ does not have an ample Fano polarization and Proposition \ref{prop:exceptionalcollectionexists} does not provide an orthogonal exceptional collection of line bundles of length $10$.
\end{Rem}

In conclusion, let $\LL$ be the admissible subcategory in Proposition \ref{prop:exceptionalcollectionexists}, and define 
\begin{equation*}
\Ku(X,\cL):=\LL^\perp=\langle \cL_1,\dots,\cL_c\rangle^\perp.
\end{equation*}
Then $\Db(X)$ admits a semiorthogonal decomposition: 
\begin{equation}\label{eqn:semicond}
\Db(X)=\langle\Ku(X,\cL),\cL\rangle.
\end{equation}

Moreover, if $X$ has an ample Fano polarization (e.g.\ if $X$ is an unnodal or a generic nodal Enriques surface), then again by Proposition \ref{prop:exceptionalcollectionexists} the admissible subcategory $\LL=\{L_1,\dots,L_{10}\}$ satisfies \eqref{cond} in the Introduction.

For an extensive discussion about some remarkable properties of $\Ku(X,\cL)$ for nodal Enriques surfaces, one can have a look at \cite{KI15}.

\subsection{Artin\textendash Mumford quartic double solids}\label{subsec:AMdoublesolids}

Let us now recall the construction of Artin-Mumford quartic double solids from \cite{KI15}.  Consider two vector spaces $V$ and $W$ of dimension $4$ and the divisor
\[
Q_s\subseteq\P(V)\times\P(W)
\]
of bidegree $(2,1)$ on $\P(V)\times\P(W)$ corresponding to a global section $s\in H^0\left(\OO_{\P(V)\times\P(W)}(2,1)\right)$. Clearly, $Q_s$ can be thought of as a family of quadrics in $\P(V)$ parameterized by $\P(W)$.  The degeneration locus of this family of quadrics is a (singular) quartic surface $D_s\subseteq\P(W)$ called a \emph{quartic symmetroid}. Consider further the (singular) double covering $Y_s\to\P(W)$ ramified over $D_s$.  For generic $s$, $D_s$ has $10$ singular points, corresponding to the quadrics of corank $2$, and we will refer to $Y_s$ as a \emph{general Artin\textendash Mumford quartic double solid}. It was explained in \cite{Cos} that for generic $s$, the $Y_s$ are precisely the Artin\textendash Mumford conic bundles constructed in \cite{AM} as examples of unirational, but not rational, conic bundles. From now on, we will assume that $s$ is general so that the Artin\textendash Mumford quartic double solid is general as well and, for simplicity, we will remove the section $s$ from the notation.

Now $Y$ is singular at precisely 10 points above the 10 singular points of $D$.  Let $Y'$ be the blow-up of $Y$ at these 10 singular points. By \cite[Lemma 3.6]{KI15}, the variety $Y'$ is the double covering of the blow-up of $\P(W)$ at the $10$ singular points of $D$ ramified over the proper preimage (and not just the proper transform) $D'$ of $D$. So we have the following diagram:
\begin{equation}\label{eqn:diablow}
\xymatrix{
	Y'\ar[r]\ar[d]_-{2:1}&Y\ar[d]^-{2:1}\\
	\text{Bl}_{\text{10 pts}}(\P(W))\ar[r]&\P(W).
}
\end{equation}
It follows that $D'$ is the blow-up of $D$ at its $10$ singular points.  As $s$ is assumed to be generic, $D'$ is a smooth K3 surface with a fixed-point free involution $\tau$. Indeed, the surface $D'$ can be seen as the zero locus of the section $s$ viewed as a global section of $W^\vee\otimes\OO_{\P(V)\times\P(V)}(1,1)$ and $\tau$ is just the restriction of the transposition of factors in $\P(V)\times\P(V)$.  The quotient $X:=D'/\langle\tau\rangle$ is then an Enriques surface which will be called the \emph{Enriques surface associated} to $Y$. These Enriques surfaces are nodal (see \cite[Remark 4.1]{KI15} and the reference therein).

We now set some more notation.  Let $e_i$ (resp. $Q_i$) be the class of the $i$-th exceptional divisor of $\text{Bl}_{\text{10 pts}}(\P(W))\to\P(W)$ (resp. $Y'\to Y$).  Each $Q_i$ is isomorphic to $\P^1\times \P^1$ because it is the exceptional divisor of the blow-up along an ordinary double point. Let $G_i=\OO_{Q_i}(-1,0)$.  By \cite[Corollary 3.8 and Lemma 3.12]{KI15}, we have a semiorthogonal decomposition
\begin{equation}\label{eqn:semiAM1}
\Db(Y')=\langle\Ku(Y'),\{G_i\}_{i=1}^{10},\OO_{Y'}(-h),\{\OO_{Y'}(-e_i)\}_{i=1}^{10},\OO_{Y'}\rangle,
\end{equation}
where $h$ is the pull-back of the hyperplane class in $\P(W)$. The subcategory $\Ku(Y')$ is defined to be the right orthogonal to all the other objects while both $\{G_i\}_{i=1}^{10}$ and $\{\OO_{Y'}(-e_i)\}_{i=1}^{10}$ are orthogonal exceptional collections. Set
\begin{equation}\label{eqn:AQDS}
\AA_{Y'}:=\langle\Ku(Y'),\{G_i\}_{i=1}^{10}\rangle
\end{equation}
and denote by $\mathsf{S}_{\AA_{Y'}}$ its Serre functor.

Note that the action on $\Db(Y')$ of the Galois involution on $Y'$ given by the double cover $Y'\to\text{Bl}_{\text{10 pts}}(\P(W))$ in \eqref{eqn:diablow} preserves the exceptional collection $\{\OO_{Y'}(-h),\{\OO_{Y'}(-e_i)\}_{i=1}^{10},\OO_{Y'}\}$ in \eqref{eqn:semiAM1} because they are pull-backs from $\text{Bl}_{\text{10 pts}}(\P(W))$. Thus the Galois involution preserves the residual category $\AA_{Y'}$. We will explain later how $\mathsf{S}_{\AA_{Y'}}$ is related to such an involution.

In \cite{KI15}, the authors exhibit an embedding of $X$ inside the Grassmannian $\mathrm{Gr}(2,V)$, providing it with a Reye polarization (see \cite[Section 2.4]{DM} for the definition), which is, in particular, an ample Fano polarization. By Proposition \ref{prop:exceptionalcollectionexists}, the Enriques surface $X$ associated to $Y$ has an explicit semiorthogonal decomposition
\[
\Db(X)=\langle\Ku(X,\LL),\LL\rangle,
\]
where $\LL=\{L_1,\dots,L_{10}\}$ is an orthogonal exceptional collection of line bundles. The following statement collects the main results in \cite{KI15} concerning $\AA_{Y'}$ and $\Ku(Y')$.

\begin{Thm}[\cite{KI15}, Corollary 3.8, Theorem 4.3]\label{thm:IK}
In the above setting:
\begin{enumerate}
\item[{\rm (1)}] There is an isomorphism of exact functors $\mathsf{S}_{\AA_{Y'}}\cong\mathsf{I}[2]$ where $\mathsf{I}$ is the non-trivial involution of $\AA_{Y'}$ (i.e.\ $\mathsf{I}\circ\mathsf{I}\cong\id_{\AA_{Y'}}$) induced by the Galois involution of $Y'$.

\item[{\rm (2)}] There is an equivalence $\Ku(Y')\cong\Ku(X,\LL)$ of Fourier\textendash Mukai type.
\end{enumerate}
\end{Thm}

This result strongly suggests that $\AA_{Y'}$ should be very much related to the derived category of $X$. In other words, it is very suggestive to guess that the correspondence between general Artin\textendash Mumford quartic double solids and associated Enriques surfaces might have a nice categorical counterpart. This was made precise by the following \cite[Conjecture 4.2]{KI15}:

\begin{Con}[Ingalls\textendash Kuznetsov]\label{conj:IK}
If $Y'$ is the blow-up of a general Artin\textendash Mumford quartic double solid $Y$ at its $10$ singular points, then there is an equivalence $\AA_{Y'}\cong\Db(X)$, where $X$ is the Enriques surface associated to $Y$.
\end{Con}

This conjecture was proved by Hosono and Takagi in \cite{HT}. As we explained in the introduction, one of the aims of this paper is to give a shorter, simpler proof (see Theorem \ref{thm:IKConj} and Theorem \ref{thm:IKConjgen} in Section \ref{subsec:mainthm1} for a more precise statement).

\section{Spherical objects in Enriques categories}\label{sec:3spherical}

The proof of Theorem \ref{thm:derived_torreli} is based on the classification of spherical objects in $\Ku(X,\cL)$, for $X$ an Enriques surface. We will explain this in a slightly more general setting which is suited to deal with the case of Artin\textendash Mumford quartic double solids as well.

Recall that if $\AA$ is a $\K$-linear triangulated category with Serre functor $\mathsf{S}_\AA$ and $d$ is a positive integer, we have the following.

\begin{Def}
An object $E\in\AA$ is \emph{$d$-spherical} if
\[
\Hom(E,E[p])\cong\begin{cases}
	\K & \mbox{if } p=0, d, \\
	0  & \mbox{otherwise}\end{cases}
\]
and $\mathsf{S}_\AA(E)\cong E[d]$.
\end{Def}

In particular, the graded $\Ext$-algebra of a $d$-spherical object is isomorphic to the cohomology of a sphere of dimension $d$.

We are interested in studying this kind of objects in categories which resemble the derived category of an Enriques surface $X$. 
 
\begin{Def}\label{def:EnriquesCat}
A $\K$-linear triangulated category $\AA$ is an \emph{$m$-Enriques category}, for $m\geq 0$ a rational number in $\frac{1}{2}\Z$, if it possesses a Serre functor $\mathsf{S}_{\AA}$ with the property that $\mathsf{S}^2_{\AA}\cong [2m]$.
\end{Def}

\begin{Rem}\label{rmk:noex}
(i) The expert reader has certainly noticed that $m$-Enriques categories are in particular fractional Calabi\textendash Yau categories of dimension $m$ (see \cite{KuzCYcat}).

(ii) If $\AA$ is an $m$-Enriques category with $m\in\Z$, then we can write $\mathsf{S}=\mathsf{I}[m]$, where $\mathsf{I}$ is an autoequivalence of $\AA$ such that $\mathsf{I}^2\cong\id$. Indeed, we can just take $\mathsf{I}:=\mathsf{S}[-m]$.

(iii) In the special case when $\AA$ is an $m$-Enriques category, $m\geq 1$ is an integer and $\mathsf{S}\cong[m]$, then, by Serre duality, $\AA$ cannot contain exceptional objects.
\end{Rem}

\begin{Ex}\label{ex:ECat}
There are several examples of $m$-Enriques categories available in the literature. We discuss some of them here. The first two are particularly relevant given the scope of this paper.

(i) If $X$ is an Enriques surface, then we already observed that the dualizing sheaf $\omega_X=\mathcal{O}_X(K_X)$ is $2$-torsion. Thus the Serre functor $\mathsf{S}_X(-)=(-)\otimes\omega_X[2]$ satisfies $\mathsf{S}_X^2\cong[4]$ and $\Db(X)$ is a $2$-Enriques category.

(ii) By the discussion in Section \ref{subsec:AMdoublesolids}, if $Y'$ is the blow-up of a general Artin\textendash Mumford quartic double solid at its $10$ singular points, then the category $\cA_{Y'}$ defined in \eqref{eqn:AQDS} is a $2$-Enriques category by Theorem \ref{thm:IK}(1).

(iii) Let $X$ be a Gushel\textendash Mukai variety of dimension $n$. It is a smooth intersection of the cone in $\mathbb{P}^{10}$ over the Grassmannian $\mathrm{Gr}(2,5)\hookrightarrow\mathbb{P}^9$ with $\mathbb{P}^{n+4}\subseteq\mathbb{P}^{10}$ and a quadric hypersurface $Q\subseteq\mathbb{P}^{n+4}$. As explained in \cite{KP}, $\Db(X)$ always contains an admissible subcategory $\cA_X$ which, by \cite[Proposition 2.6]{KP}, is a $2$-Enriques category. If $n$ is even, then $\mathsf{S}_{\cA_X}=[2]$ and thus we are in the situation described in Remark \ref{rmk:noex} (iii). If $n$ is odd, then $\mathsf{S}_{\cA_X}$ is not a shift.

(iii) Other examples of $m$-Enriques categories with $m\neq 2$ can be found in \cite{KuzCYcat}.
\end{Ex}

We now work in the following setting:

\begin{Set}\label{setup}
$\AA$ is an $m$-Enriques category with a semiorthogonal decomposition $$\AA=\langle\cK,L_1,\dots,L_N\rangle,$$ where $\KK$ is admissible, $N\geq 1$ is an integer, and $\cL:=\{L_1,\dots,L_N\}$ is an orthogonal exceptional collection. 

\end{Set}

Note that since $\cA$ contains at least one exceptional object, by Remark \ref{rmk:noex} (iii), the Serre functor of $\cA$ cannot be isomorphic to a shift.

\begin{Ex}\label{ex:EnriquesCat}
By Sections \ref{subsec:Enriques} and \ref{subsec:AMdoublesolids} and Example \ref{ex:ECat} (i) and (ii), $\AA=\Db(X)$, where $X$ is either an Enriques surface, and $\AA=\AA_{Y'}$, where $Y'$ is the blow-up of an Artin\textendash Momford quartic double solid, are both as in Setup \ref{setup}. Indeed, in both cases, the $2$-Enriques categories in Example \ref{ex:ECat} (i) and (ii) contain $10$ orthogonal exceptional objects.
\end{Ex}

Suppose that we are in Setup \ref{setup}, and let us simplify notation by setting $\mathsf{S}:=\mathsf{S}_{\AA}$ and letting $\kappa\colon\cK\to\cA$ be the natural inclusion with left adjoint $\kappa^*$. By Serre duality, we have $\Hom(L_i,\mathsf{S}(L_i))\cong\K$, so we can define $S_i\to L_i$ to be the cocone of the unique non-zero (up to scaling) morphism $L_i\to \mathsf{S}(L_i)$.  These objects fit in the distinguished triangle
\begin{equation}\label{eq:mutation1}
S_i\to L_i\to\mathsf{S}(L_i).
\end{equation}
All the objects in \eqref{eq:mutation1} are clearly contained in $\cA$. We claim that more is true

\begin{Lem}\label{lem:inK}
For $i=1,\dots,N$, we have $S_i\in\KK$.
\end{Lem}

\begin{proof}
We must show that $\RHom(L_j,S_i)=0$ for $j=1,\dots,N$.  If $j\neq i$, then $\RHom(L_j,L_i)\cong\RHom(L_j,\mathsf{S}(L_i))=0$, by assumption (b) in Setup \ref{setup} and Serre duality.  Thus $\RHom(L_j,S_i)\cong0$ for $j\ne i$.   On the other hand, the morphism $L_i\to\mathsf{S}(L_i)$ in \eqref{eq:mutation1} yields an isomorphism $\RHom(L_i,L_i)\cong\RHom(L_i,\mathsf{S}(L_i))$. Thus $\RHom(L_i,S_i)\cong 0$ as well.
\end{proof}

It then follows that in-fact we have 
\begin{equation}\label{eqn:difdef}
S_i=\kappa^*(\mathsf{S}(L_i))[-1],
\end{equation}
for $i=1,\dots,N$.

The following result summarizes the important properties of the $S_i$.

\begin{Lem}\label{lem:Twy}
In Setup \ref{setup}, the object $S_i$ is $(2m-1)$-spherical in $\cK$, for all $i=1,\dots,N$. Moreover, $\RHom(S_i,S_j)\cong 0$ and thus $S_i\not\cong S_j[k]$ for all $i\neq j\in\{1,\dots,N\}$ and for any $k\in\Z$. 
\end{Lem}

\begin{proof}
We begin by proving the first condition for being $(2m-1)$-spherical as well as the second statement of the lemma.  

Since $\mathsf{S}$ is an autoequivalence, we get isomorphisms of graded vector spaces
\[
\RHom(\mathsf{S}(L_i),\mathsf{S}(L_j))\cong\RHom(L_i,L_j)\cong\K^{\delta_{ij}},
\]
for $i,j=1,\dots,N$. We also get the isomorphisms of graded vector spaces
\begin{align}\label{eqn:orto2}
\RHom(\mathsf{S}(L_i),L_j)&\cong\RHom(\mathsf{S}^2(L_i),\mathsf{S}(L_j))\\
\nonumber
&\cong\RHom(L_i[2m],\mathsf{S}(L_j))\\\nonumber
&\cong\RHom(L_j,L_i[2m])^\vee\\\nonumber
&\cong\K^{\delta_{ij}}[-2m],
\end{align}
where the first, second, and third isomorphisms follow from $\mathsf{S}$ being an autoequivalence, the fact that $\mathsf{S}^2\cong[2m]$, and Serre duality, respectively.  These two remarks show that if we apply $\RHom(\mathsf{S}(L_i),-)$ to \eqref{eq:mutation1}, we get an isomorphism of graded vector spaces
\begin{equation}\label{eqn:orto}
\RHom(\mathsf{S}(L_i),S_j)\cong\K^{\delta_{ij}}[-1]\oplus\K^{\delta_{ij}}[-2m].
\end{equation}
Hence, if we apply $\RHom(-,S_j)$ to \eqref{eq:mutation1} and we use that $S_j\in\cK\subseteq L_i^\perp$ by Lemma \ref{lem:inK}, we get an isomorphism of graded vector spaces
\[
\RHom(S_i,S_j)\cong\K^{\delta_{ij}}\oplus\K^{\delta_{ij}}[1-2m],
\]
as required.

It remains to show that $\mathsf{S}_\KK(S_i)\cong S_i[2m-1]$, for $i=1,\dots,N$. This easily follows from the following chain of isomorphisms
\begin{align*}
\mathsf{S}_{\cK}^{-1}(S_i)&\cong\mathsf{L}_{\langle L_1,\dots,L_{N}\rangle}(\mathsf{S}^{-1}(S_i))\\ &\cong\mathsf{L}_{L_i}(\mathsf{S}^{-1}(S_i)) \\
&\cong\mathsf{L}_{L_i}(\mathrm{Cone}(\mathsf{S}^{-1}(L_i)[-1]\to L_i[-1])) \\
&\cong\mathsf{L}_{L_i}(\mathsf{S}^{-1}(L_i))\\
&\cong S_i[-2m+1].
\end{align*}
Observe that the first isomorphism follows from \cite[Lemma 2.6]{KuzCYcat}, while for the second one we used \eqref{eqn:LeftMutationExceptional}, the fact that $\mathsf{S}$ is an equivalence, and \eqref{eqn:orto}. The third one and the last one use \eqref{eq:mutation1} (to which we apply $\mathsf{S}^{-1}$) and \eqref{eqn:orto2}. The fourth isomorphism follows from $\mathsf{L}_{L_i}(L_i)\cong 0$.
\end{proof}

\begin{Rem}\label{rmk:isomun}
(i) If we combine Serre duality and Lemma \ref{lem:inK}, then we get that $$\Hom(S_i,\mathsf{S}(L_i)[p-1])\cong\Hom(L_i,S_i[1-p])^\vee\cong0,$$ for all $p$. So by applying $\RHom(S_i,-)$ to \eqref{eq:mutation1}, we see that there is a unique (up to scalars) morphism $f_i\colon S_i\to L_i$.  For the same reason, the composition with $f_i$ provides an isomorphism $\Hom(E,S_i) \cong \Hom(E,L_i)$,  for all $E\in\cK$.

(ii) Let us reconsider the embedding $\kappa\colon\cK\to\cA$ realizing $\cK$ as an admissible subcategory. An easy exercise with Serre duality shows that its right adjoint $\kappa^!$ is described by the formula
\[
\kappa^!=\mathsf{S}_\cK\circ\kappa^*\circ\mathsf{S}_\cA^{-1},
\]
where $\mathsf{S}_\cK$ and $\mathsf{S}_\cA$ are the Serre functors of $\cK$ and $\cA$, respectively. Thus, as an application of \eqref{eqn:difdef} and Lemma \ref{lem:Twy}, we get the isomorphisms
\[
\kappa^!(L_i)\cong \mathsf{S}_\cK\circ\kappa^*\circ\mathsf{S}_\cA^{-1}(L_i)\cong\mathsf{S}_\cK\circ\kappa^*\circ\mathsf{S}_\cA(L_i)[-2m]\cong\mathsf{S}_\cK(S_i)[-2m+1]\cong S_i,
\]
for $i=1,\dots,N$. For the second isomorphism we used that $\mathsf{S}_\cA^2\cong[2m]$.
\end{Rem}

The lemma shows how to construct $(2m-1)$-spherical objects in $\cK$. We want to prove now that the list is complete.

\begin{Prop}\label{prop:10Ts}
Let $\AA$ be an $m$-Enriques category with a semiorthogonal decomposition $$\AA=\langle\cK,L_1,\dots,L_N\rangle,$$ where $\KK$ is admissible, $N\geq 1$ is an integer, and $\cL:=\{L_1,\dots,L_N\}$ is an orthogonal exceptional collection. Let $F$ be a $(2m-1)$-spherical object in $\cK$. If $m\geq\frac{3}{2}$, then $F\cong S_i[k]$ for some $i=1,\dots, N$ and some integer $k$.
\end{Prop}

\begin{proof}
By \cite[Lemma 2.6]{KuzCYcat} and since $\mathsf{S}^2(F)\cong F[2m]$ (as in Definition \ref{def:EnriquesCat}), the isomorphism $\mathsf{S}_{\cK}(F)\cong F[2m-1]$ is equivalent to the existence of an isomorphism $F\cong\mathsf{L}_{\langle L_1,\dots,L_{N}\rangle}(\mathsf{S}(F))[-1]$. Define the object $P:=\bigoplus_{1\leq i\leq N} L_i$. It is \emph{polyexceptional} in the sense that it is a direct sum of objects forming an orthogonal exceptional collection and so its endomorphism algebra is such that $$\RHom(P,P)\cong\K^N,$$. It follows from \eqref{eqn:LeftMutationExceptional} that $F$ sits in a distinguished triangle
\begin{equation}\label{eq:lmut}
F\to P\otimes_{\K^N} \RHom(P,\mathsf{S}(F))\to \mathsf{S}(F).
\end{equation}

If we apply $\RHom(F,-)$ to \eqref{eq:lmut}, then, by Serre duality and taking into account that $F$ is $(2m-1)$-spherical, we get a distinguished triangle of graded $\K$-vector spaces
\begin{equation}\label{eq:trialm}
\K\oplus\K[-2m+1]\to\RHom(F,P)\otimes_{\K^N}\RHom(F,P)^\vee\to\K\oplus\K[2m-1].
\end{equation}
Hence, since $2m\geq 3$, the graded vector space $V:=\RHom(F,P)\otimes_{\K^N}\RHom(F,P)^\vee$ is $4$-dimensional.  Moreover, since $\RHom(P,P)$ consists only of componentwise morphisms, $V$ decomposes as
\[
V\cong\bigoplus_{i=1}^N V_i\otimes V_i^\vee,
\]
where $V_i:=\RHom(F,L_i)$. For dimension reasons, we can only have two possibilities: either there are four non-trivial $1$-dimensional $V_i$ or there is only one non-trivial $2$-dimensional $V_i$. But in the former case, as $V_i\otimes V_i^\vee$ is concentrated in degree $0$ for each $i$, we get a contradiction since $V$ has 2-dimensional degree $0$ piece by \eqref{eq:trialm}. Thus there must be only one non-trivial $V_i$ which then must satisfy
\[
V_i\cong\K[k-2m+1]\oplus\K[k],
\]
for some $k\in\Z$. Thus, up to reordering the objects in $\cL$ and shifting, we can assume $i=1$ and $k=0$. To simplify the notation, for the rest of the proof, we set $L:=L_1$. Hence
\begin{equation}\label{eqn:computations}
\Hom(F,L[p])\cong\begin{cases}
\K & \text{if $p=0,2m-1$,}\\
0 & \text{otherwise.}
\end{cases}
\end{equation}
Then \eqref{eq:lmut} gets the following simplified form, after a rotation:
\begin{equation}\label{eq:fkf}
L[-1]\oplus L[2m-2]\rightarrow \mathsf{S}(F)[-1]\rightarrow F.
\end{equation}
We will show that $F\cong S_1$, as required.

To that end, let $C$ be the cone of the unique (up to scalars) morphism $L[2m-2]\to \mathsf{S}(F)[-1]$ so that we get the following distinguished triangle
\begin{equation}
L[2m-2]\to \mathsf{S}(F)[-1]\to C.
\label{eq:conec}
\end{equation}
Note that it can put in the following commutative diagram of distinguished triangles
\begin{center}
	\begin{tikzcd}
		L[2m-2] \arrow{d} \arrow{r}
		& \mathsf{S}(F)[-1] \arrow{r} \arrow{d}{\id} & C\arrow{d} \\
		L[-1]\oplus L[2m-2] \arrow{r} \arrow{d}& \mathsf{S}(F)[-1]
		\arrow{d} \arrow{r} & F \arrow{d}\\
		L[-1] \arrow{r} & 0\arrow{r} & L
	\end{tikzcd}
\end{center}
in view of the octahedron axiom. Thus $C$, $F$ and $L$ sit in a distinguished triangle
\begin{equation}\label{eq:trafin}
C\stackrel{k_1}{\longrightarrow} F\longrightarrow L.
\end{equation}

By applying $\RHom(L,-)$ to \eqref{eq:trafin} (resp. $\RHom(-,L)$ to \eqref{eq:conec}) and using the fact that $F\in\cK$ and Serre duality, we see that
\begin{equation}\label{eqn:cases}
\Hom(L,C[j])\cong\begin{cases}
\K & \text{if $j=1$} \\
0 & \text{otherwise}
\end{cases}\qquad\text{and}\qquad\Hom(C,L[j])\cong\begin{cases}
\K & \text{if $j=2m-1$} \\
0 & \text{otherwise.}
\end{cases}
\end{equation}
This and \eqref{eqn:computations} show that the morphism $F\to L$ in \eqref{eq:trafin} is not trivial. Otherwise $C\cong F\oplus L[-1]$, contradicting the second part of \eqref{eqn:cases}.

We claim that to show $F\cong S_1$, it suffices to show that $C\cong \mathsf{S}(L)[-1] $. Indeed, observe first that $F\not\cong\mathsf{S}(L)[-1]\oplus L$ because $\Hom(F,F)\cong\K$ as $F$ is a $(2m-1)$-spherical object. Thus, in view of \eqref{eq:trafin}, $C\cong\mathsf{S}(L)[-1]$ would imply that $F$ is isomorphic to the cone of the (unique up to scalars) non-trivial map $L[-1]\to \mathsf{S}(L)[-1]$, which is $S_1$ by \eqref{eq:mutation1}.

To show $C\cong \mathsf{S}(L)[-1] $, we first prove that it is exceptional and then produce non-trivial morphisms $f\colon \mathsf{S}(L)[-1] \to C$ and $g\colon C\to \mathsf{S}(L)[-1] $ whose composition $g\circ f$ is non-zero. As $\mathsf{S}(L)[-1]$ is exceptional, up to scaling $f$ or $g$, we may assume that $g\circ f=\id$.  But then $\mathsf{S}(L)[-1]$ must be a direct summand of $C$, contradicting the fact that $C$ is exceptional. Hence $C\cong\mathsf{S}(L)[-1]$.  

By Serre duality and \eqref{eqn:cases} we conclude that 
\begin{align*}&\Hom(\mathsf{S}(L)[-1],C)^\vee\cong\Hom(C,\mathsf{S}^2(L)[-1])\cong\Hom(C,L[2m-1])\cong\K,\\
&\Hom(C,\mathsf{S}(L)[-1])^\vee\cong\Hom(L,C[1])\cong\K.\end{align*}
Thus we let $f\in\Hom(\mathsf{S}(L)[-1],C)$ and $g\in\Hom(C,\mathsf{S}(L)[-1])$ be the unique (up to scalars) non-zero morphisms.

Before we prove that $g\circ f\ne 0$, we prove that $C$ is exceptional, recording a useful result along the way.  We claim indeed that $$\Hom(C,\mathsf{S}(F)[j])=0,\;\text{for all}\; j\in\Z,\qquad\text{and}\qquad\Hom(C,C[j])\cong\begin{cases}
\K & \text{if $j=0$} \\
0 & \text{otherwise.}
\end{cases}$$ To that end, applying $\RHom(F,-)$ to \eqref{eq:trafin}, we get $\Hom(F,C[j])=0$ for every $j\leq 2m-2$. By Serre duality, we get $\Hom(C,\mathsf{S}(F)[j])=0$ for every $j\geq 2-2m\geq -1$ (note that $m\geq \frac{3}2$).
At this point, we apply $\RHom(-,\mathsf{S}(F))$ to \eqref{eq:conec} and get $\Hom(C,\mathsf{S}(F)[j])=0$ for every $j\leq -2$. Therefore,  we conclude that $\Hom(C,\mathsf{S}(F) [j])=0$, for every $j\in \Z$. Finally, if we apply $\RHom(C,-)$ to \eqref{eq:conec}, we get $C$ is exceptional.

It remains to show that $g\circ f$ is non-zero. Let $k_2\in \Hom(\mathsf{S}(F)[-1],\mathsf{S}(L)[-1])\cong\K$ be the unique (up to scalars) non-trivial morphism, and similarly, let $k_3$ be the unique (up to scalars) non-trivial morphism in $\Hom(\mathsf{S}(L)[-1],F)\cong\K$.  We claim that $k_3\circ k_2\ne 0$.  Indeed,  applying $\RHom(\mathsf{S}(-),F)$ to \eqref{eq:fkf}, we get an exact sequence 
\begin{gather*}
\Hom(\mathsf{S}^2(F)[-1],F)\to\Hom(\mathsf{S}(L)[-1],F)\oplus\Hom(\mathsf{S}(L)[2m-2],F)\to\\
\to\Hom(\mathsf{S}(F)[-1],F)\to\Hom(\mathsf{S}^2(F)[-2],F).
\end{gather*}
Now for $i=0,1$, $\Hom(\mathsf{S}^2(F)[-(i+1)],F)\cong\Hom(F[2m-(i+1)],F)\cong 0$ because $F$ is $(2m-1)$-spherical and $\mathsf{S}^2\cong[2m]$. Moreover, as $m\geq\frac{3}{2}$, it follows from \eqref{eqn:computations} that
\[
\Hom(\mathsf{S}(L)[2m-2],F)\cong\Hom(F,\mathsf{S}^2(L)[2m-2])^\vee\cong\Hom(F,L[4m-2])^\vee\cong 0.
\]
Thus precomposition with $k_2$ induces an isomorphism
\begin{equation}\label{eqn:iso112}
\Hom(\mathsf{S}(L)[-1],F)\xrightarrow{\sim}\Hom(\mathsf{S}(F)[-1],F).
\end{equation}
In particular, $k_3\ne 0$ has non-trivial image, i.e. $k_3\circ k_2\ne 0$.

Applying $\RHom(-,C)$ to \eqref{eq:conec} and using \eqref{eqn:cases}, we see that there is a unique (up to scalars) non-trivial morphism $k_4\in\Hom(\mathsf{S}(F)[-1],C)\cong\K$.  The same argument as in the previous paragraph, by applying $\RHom(-,\mathsf{S}(L)[-1])$ to \eqref{eq:conec}, yields that precomposition with $k_4$ induces an isomorphism  \begin{equation}\label{eqn:iso113}
\Hom(C,\mathsf{S}(L)[-1])\xrightarrow{\sim}\Hom(\mathsf{S}(F)[-1],\mathsf{S}(L)[-1]).
\end{equation}

Now we claim that the composition of non-trivial morphisms
\[
\mathsf{S}(F)[-1]\stackrel{k_4}{\longrightarrow} C\stackrel{g}{\longrightarrow} \mathsf{S}(L)[-1]\stackrel{k_3}{\longrightarrow} F
\]
is a non-trivial morphism. Indeed, by \eqref{eqn:iso113} $g\circ k_4$ is non-zero, so there exists $0\neq\mu\in\K$ such that $0\neq g\circ k_4=\mu k_2$. But then
\[
k_3\circ g\circ k_4=\mu k_3\circ k_2\neq 0.
\]
In particular, it follows that the composition $k_3\circ g$ is non-zero.  

Applying $\RHom(C,-)$ to \eqref{eq:trafin}, we see that $\Hom(C,F)\cong\K$ from a simple computation using \eqref{eqn:cases} and the fact that $\Hom(C,C)\cong \K$. Thus there exists $0\neq\lambda\in\K$ such that
\[
 k_3\circ g=\lambda k_1.
\]

Finally, applying $\RHom(\mathsf{S}(L)[-1],-)$ to \eqref{eq:trafin}, we get the isomorphism
\begin{equation}\label{eqn:isoimport}
\Hom(\mathsf{S}(L)[-1],C)\xrightarrow{\sim}\Hom(\mathsf{S}(L)[-1],F),
\end{equation}
defined by the postcomposition by the non-trivial map $k_1$. In particular, since $f\ne 0$, we have $k_1\circ f\ne 0$.  But then 
\[0\ne \lambda (k_1\circ f)=(\lambda k_1)\circ f=k_3\circ g\circ f,
\]
so $g\circ f\ne 0$, as required.
\end{proof}

\begin{Rem}\label{rmk:general2}
It would be interesting to study whether Proposition \ref{prop:10Ts} can be generalized to situations where one has semiorthogonal decompositions as in Proposition \ref{prop:exceptionalcollectionexists} when the surface does not have an ample Fano polarization. Unfortunately, the proof provided above does not apply in this setting.
\end{Rem}

\section{Proof of the main results}\label{sec:proof}

The proof of the main results in this paper is explained in Section \ref{subsec:mainthm1} and it follows from a more general result which will be proved in the next section.

\subsection{A general extension result}\label{sec:genextresult}

Let $Z_1$ and $Z_2$ be smooth projective varieties over $\K$ admitting admissible subcategories $\AA_{Z_i}$ which fall under Setup \ref{setup}.  That is, the $\AA_{Z_i}$ are $m$-Enriques and admit semiorthogonal decompositions $$\AA_{Z_i}=\langle\Ku(Z_i),L^i_1,\dots,L^i_{N_i}\rangle,$$ where $\Ku(Z_i)$ is admissible in $\AA_{Z_i}$, $N_i\geq 1$ is an integer, and $\cL_i:=\{L^i_1,\dots,L^i_{N_i}\}$ is an orthogonal exceptional collection.  We can then prove the following general result.

\begin{Thm}\label{thm:gen}
Under the assumptions above, let $\mathsf{F}\colon \Ku(Z_1)\to\Ku(Z_2)$ be an equivalence which is of Fourier\textendash Mukai type. Then $N_1=N_2=N$ and there exists a Fourier\textendash Mukai functor $\Phi_{\tilde{\cE}}\colon\Db(Z_1)\to\Db(Z_2)$ such that
\begin{itemize}
\item[{\rm (1)}] $\Phi_{\tilde{\cE}}|_{\Ku(Z_1)}\cong\mathsf{F}$;
\item[{\rm (2)}] $\Phi_{\tilde{\cE}}|_{\cA_{Z_1}}\colon\cA_{Z_1}\to\cA_{Z_2}$ is an equivalence; and
\item[{\rm (3)}] Up to reordering and shift, $\Phi_{\tilde{\cE}}(L_i^1)\cong L_i^2$, for all $i=1,\dots,N$. 
\end{itemize}
\end{Thm}

\begin{proof}
We first observe that $N_1=N_2=N$. Indeed, by Lemma \ref{lem:Twy} and Proposition \ref{prop:10Ts} applied to $\AA_{Z_i}$, $\Ku(Z_i)$ has to contain, up to shift, exactly $N_i$ pairwise non-isomorphic $(2m-1)$-spherical objects. Since such a number is invariant under equivalence and $\Ku(Z_1)\cong\Ku(Z_2)$, we must have $N_1=N_2$.

 Now let $\cE\in\Db(Z_1\times Z_2)$ be such that $\mathsf{F}=\Phi_{\EE}|_{\Ku(Z_1)}$.  We would like to apply Proposition \ref{prop:extension2} with $$\AA=\AA_{Z_1}\qquad X=Z_1,\qquad Y=Z_2,\qquad \BB=\AA_{Z_2}.$$ By Remark \ref{rmk:FMfun0onorth}, we may choose $\cE$ such that $\Phi_\EE|_{\!^\perp\!\Ku(Z_1)}=0$ without changing $\Phi_\EE|_{\Ku(Z_1)}$.  Thus we can suppose without loss of generality that $\Phi_\EE$ is as in Proposition \ref{prop:extension2} and that assumption (a) of the proposition is verified. Consider, for any $j=1,2$ and any $i=1,\ldots,N$, the $(2m-1)$-spherical object $S_i^j\in\Ku(Z_j)$ defined by the distinguished triangle
\[
S_i^j\to L_i^j\to\mathsf{S}_{\AA_{Z_j}}(L_i^j).
\]
By Proposition \ref{prop:10Ts}, we can assume $\mathsf{F}(S^1_i)\cong S^2_i$, up to reordering and shift, for all $i=1,\dots,N$.

Denote by $$\zeta_i\colon\Ku(Z_i)\hookrightarrow\Db(Z_i)\qquad\text{and}\qquad\kappa_i\colon\Ku(Z_i)\into\AA_{Z_i}$$ the relevant embeddings for $i=1,2$.  Note that by Remark \ref{rmk:isomun} (ii), the object $\kappa_j\kappa_j^!(L^j_i)\cong S_i^j$.  An analogous computation shows that $\zeta_j\zeta_j^!(L^j_i)\cong S_i^j$, so we get an isomorphism 
$$\mathsf{F}(\zeta_1\zeta_1^!(L^1_i))\cong\mathsf{F}(S_i^1)\cong S_i^2\cong\zeta_2\zeta_2^!(L^2_i).$$
Thus assumption (b) of Proposition \ref{prop:extension2} is also satisfied if we set
\[
X=Z_1\qquad Y=Z_2\qquad\cA=\Ku(Z_1)\qquad\BB=\Ku(Z_2)\qquad E=L_1^1\qquad F=L_1^2.
\] 

Therefore, by Proposition \ref{prop:extension2}, we get a Fourier\textendash Mukai type functor $\Phi_{\tilde\EE_1}$ such that
\begin{gather*}
    \Phi_{\tilde\EE_1}|_{\langle\Ku(Z_1),L_1^1\rangle}\colon \langle\Ku(Z_1),L_1^1\rangle\isomor\langle\Ku(Z_2),L_1^2\rangle,\qquad
\Phi_{\tilde\EE_1}|_{\Ku(Z_1)}\cong\mathsf{F},\qquad\Phi_{\tilde\EE_1}(L_1^1)\cong L_1^2,\;\text{and}\\
\Phi_{\tilde\EE_1}(\!^\perp\!\langle\Ku(Z_1),L_1^1\rangle)\cong 0.
\end{gather*} 
The argument proceeds inductively. To simplify the notation slightly, let us set $L_i:=L_i^1$ and $M_i:=L_i^2$. For $k\geq 2$, we assume that we have a Fourier\textendash Mukai kernel $\tilde\EE_{k-1}$ such that $$\Phi_{\tilde\EE_{k-1}}|_{\langle\Ku(Z_1),L_1,\dots,L_{k-1}\rangle}\colon\langle\Ku(Z_1),L_1,\dots,L_{k-1}\rangle\to \langle\Ku(Z_2),M_1,\dots,M_{k-1}\rangle$$ is an equivalence satisfying
$$\Phi_{\tilde \EE_{k-1}}|_{\Ku(Z_1)} \cong \Phi_{\tilde\EE_{k-2}}|_{\Ku(Z_1)}\qquad\text{and}\qquad\Phi_{\tilde \EE_{k-1}}(L_{i})\cong M_{i},$$  for $i=1,\dots,k-1$, where we set $\tilde\EE_0:=\EE$. To proceed to stage $k$, we apply again Proposition \ref{prop:extension2} with
\[
\cA=\langle\Ku(Z_1),L_1,\dots,L_{k-1}\rangle,\qquad\cB=\langle\Ku(Z_2),M_1,\dots,M_{k-1}\rangle,\qquad E=L_k,\qquad F=M_k.
\]
Here it is important to note that, since the exceptional objects $L_1,\ldots,L_k$ (resp.\ $M_1,\ldots,M_k$) are orthogonal, by Remark \ref{rmk:isomun} (ii) the object $\alpha\alpha^!(L_k)\cong S_k^1$ (resp. $\beta\beta^!(M_k)\cong S_k^2$) is actually contained in $\Ku(Z_1)$ (resp. $\Ku(Z_2)$) and not just in $\cA$ (resp. $\BB$).  Here $\alpha\colon\cA\hookrightarrow\Db(Z_1)$ and $\beta\colon\BB\into\Db(Z_2)$ are the admissible embeddings. Hence $\Phi_{\tilde\EE_{k-1}}$ satisfies the assumptions of Proposition \ref{prop:extension2} by induction, and we can produce $\Phi_{\tilde\EE_{k}}$.
	Continuing in this way, by Proposition \ref{prop:extension2}, we get an equivalence $\Phi_{\tilde\EE_{N}}\colon\Db(\cA_{Z_1})\isomor\Db(\cA_{Z_2})$ and we just set $\tilde\EE:=\tilde\EE_N$.
\end{proof}

\subsection{Proof of Theorems \ref{thm:derived_torreli} and \ref{thm:IKConj}}\label{subsec:mainthm1}

As for Theorem \ref{thm:derived_torreli}, we can actually prove the following more general version of it.

\begin{Thm}\label{thm:derived_torreligen}
Let $X_1$ and $X_2$ be Enriques surfaces. Then the following are equivalent:
\begin{itemize}
\item[{\rm (i)}] There is a semiorthogonal decomposition
$\Db(X_i)=\langle \Ku(X_i,\cL_i),\cL_i\rangle,$
satisfying \eqref{cond}, for $i=1,2$ and an exact equivalence $$\mathsf{F}: \Ku(X_1,\cL_1)\xrightarrow{\sim}\Ku(X_2,\cL_2)$$ of Fourier\textendash Mukai type;
\item[{\rm(ii)}] $X_1\cong X_2$ and either $\Db(X_1)$ or $\Db(X_2)$ has a semiorthogonal decomposition satisfying condition \eqref{cond}.
\end{itemize}
\end{Thm}

\begin{proof}
The proof of the fact that (i) implies (ii) amounts to showing that, under the assumptions in (i), $X_1\cong X_2$. For this, we just apply Theorem \ref{thm:gen} with $$Z_i=X_i,\qquad \AA_{Z_i}=\Db(X_i),\qquad N=10,\;\text{and}\qquad \Ku(Z_i)=\Ku(X_i,\LL_i),\LL_i\rangle,$$
where $\LL_i=\{L_1^i,\dots,L_{10}^i\}$ is as in \eqref{cond}, for $i=1,2$ (see Example \ref{ex:EnriquesCat}). Indeed, the exact equivalence $\mathsf{F}\colon\Ku(X_1,\LL_1)\to\Ku(X_2,\LL_2)$ of Fourier-Mukai type extends to an equivalence $\Db(X_1)\cong\Db(X_2)$. By Theorem \ref{thm:dertordercat}, we deduce that $X_1\cong X_2$.

On the other hand, assume that there exists an isomorphism $f\colon X_1\to X_2$ and a semiorthogonal decomposition for $\Db(X_2)$ satisfying \eqref{cond}. Since $f^*\colon\Db(X_2)\to\Db(X_1)$ is an equivalence, we can take on $\Db(X_1)$ the semiorthogonal decomposition satisfying \eqref{cond} which is the image of the given one on $\Db(X_2)$ under $f^*$. Clearly, such a semiorthogonal decomposition on $\Db(X_1)$ satisfies \eqref{cond}. Finally, $\mathsf{F}:=(f^*)^{-1}|_{\Ku(X_1,\cL_1)}\colon\Ku(X_1,\cL_1)\xrightarrow{\sim}\Ku(X_2,\cL_2)$ is an exact equivalence of Fourier\textendash Mukai type by construction. The argument when $\Db(X_1)$ satisfies such a property is identical using the (quasi-)inverse of $f^*$ instead of $f^*$.
\end{proof}

\begin{Rem}\label{rmk:general3}
One way to generalize Theorem \ref{thm:derived_torreligen} to non-generic Enriques surfaces could be to consider semiorthogonal decompositions as in Proposition \ref{prop:exceptionalcollectionexists} when the surface does not have an ample Fano polarization. Let us briefly sketch the idea which was suggested to us by A.\ Perry. Indeed, one might try to deform $X_1$, $X_2$ and the equivalence $\mathsf{F}\colon\Ku(X_1,\cL_1)\to\Ku(X_2,\cL_2)$ to the generic case. In the generic case, we can then apply Theorem \ref{thm:derived_torreligen} and conclude that the deformed Enriques surfaces are isomorphic. The separatedness of the moduli space of (polarized) Enriques surfaces should then allow us to conclude that $X_1\cong X_2$.

Unfortunately, while in the generic case first order deformations of $X_i$ coincide with first order deformations of the subcategory $\Ku(X_i,\cL)$, this seems not to be the case for nodal non-generic Enriques surfaces. Thus one needs to add some natural assumptions on $\mathsf{F}$. Namely, we need the equivalence to preserve commutative first order deformations of $X_1$ and $X_2$. This will be investigated in future work.
\end{Rem}

\medskip

The proof of Theorem \ref{thm:IKConj} is along the same lines. Indeed, let $Y'$ be the blow-up at the $10$ singular points of a general Artin\textendash Mumford quartic double solid $Y$ with associated Enriques surface $X$. Consider the semiorthogonal decompositions of the $2$-Enriques categories (see Section \ref{subsec:AMdoublesolids} and Example \ref{ex:EnriquesCat})
\[
\AA_{Y'}:=\langle\Ku(Y'),\{G_i\}_{i=1}^{10}\rangle\qquad\Db(X)=\langle\Ku(X,\LL),L_1\dots,L_{10}\rangle
\]
discussed in Section \ref{subsec:AMdoublesolids}. Moreover, we know that there is an exact equivalence $\Ku(X,\LL)\cong\Ku(Y')$ of Fourier\textendash Mukai type (see Theorem \ref{thm:IK} (2)). We then apply Theorem \ref{thm:gen}.

In particular, we get the following more precise version of Theorem \ref{thm:IKConj}.

\begin{Thm}\label{thm:IKConjgen}
Under our assumptions, there is an exact equivalence $\Db(X)\cong\AA_{Y'}$ induced by a Fourier\textendash Mukai functor $\Phi_\EE\colon\Db(X)\to\Db(Y')$ such that
\begin{itemize}
\item[{\rm (1)}] $\Phi_\EE|_{\Ku(X,\cL)}\colon\Ku(X,\cL)\to\Ku(Y'))$ is the exact equivalence in Theorem \ref{thm:IK} (2), and
\item[{\rm (2)}] Up to reordering and shifts, $\Phi_\EE(L_i)=G_i$, for $i=1,\dots,10$.
\end{itemize}
\end{Thm}

In a sense, this result proves a stronger version of Conjecture \ref{conj:IK} since the equivalence we construct is automatically compatible with the semiorthogonal decompositions of $\Db(X)$ and $\cA_{Y'}$.


\bigskip

{\small\noindent{\bf Acknowledgements.} This work began in a problem-solving working group as part of the \emph{Workshop on  ``Semiorthogonal decompositions, stability conditions and sheaves of categories''} held at the University of Toulouse in 2018. It is our pleasure to thank this institution and the organizers of the workshop for the very stimulating atmosphere. Marcello Bernardara, Daniele Faenzi, Sukhendu Mehrotra, and Franco Rota were part of our working group in Toulouse. We would like to warmly thank them for the very interesting discussions we had with them. We are also grateful to Arend Bayer, Alberto Canonaco, Andreas Hochenegger, Alexander Kuznetsov, Christian Liedtke, Emanuele Macr\`i, Riccardo Moschetti, Alex Perry, Giorgio Scattareggia and Sofia Tirabassi for the many suggestions and comments we received from them. Finally, we would like to thank the anonymous referee whose suggestions and comments helped us greatly improve the exposition in the paper.}


\end{document}